\documentclass[a4paper,11pt,english]{article}
\usepackage[utf8]{inputenc}
\usepackage[T1]{fontenc}
\usepackage{babel}

\usepackage{fullpage}
\usepackage{units}

\usepackage{array}
\usepackage{float}
\usepackage{makecell}

\usepackage{dsfont,amsmath,amsthm,amsfonts,amssymb,upgreek,eufrak}
\usepackage{caption}

\usepackage{graphicx}
\usepackage{multirow}
\usepackage{color}
\usepackage[all]{xy}

\usepackage{url}
\usepackage{cite}
\usepackage[colorlinks=true,pagebackref=true
,linkcolor=blue,bookmarks=true,bookmarksopen=true,citecolor=blue]{hyperref}

\makeatletter
\newcommand{\xRightarrow}[2][]{\ext@arrow 0359\Rightarrowfill@{#1}{#2}}
\makeatother

\usepackage{tikz}
\usetikzlibrary{arrows}
\usepackage{pgf, tikz}
\usetikzlibrary{arrows, automata}

\usepackage{algorithm}
\usepackage{algorithmic}

\newtheorem{thm}{Theorem}[section]
\newtheorem{defi}{Definition}[section]
\newtheorem{rem}{Remark}[section]
\newtheorem{prop}{Proposition}[section]
\newtheorem{lem}{Lemma}[section]
\newtheorem{coro}{Corollary}[section]

\newtheorem{hyp}{Assumption}[section]
\newtheorem{conj}{Conjecture}

\newenvironment{acknowledgments}
{\noindent\itshape\small\textbf{Acknowledgments:}}%
{}

\usepackage{color}
\usepackage{cite}

\definecolor{ao}{rgb}{0.0, 0.5, 0.0}

\usepackage{bigints}

\usepackage[blocks]{authblk}

\author[1]{Dovgal Sergey}
\author[1]{Offret Yoann}

\affil[1]{Institut de Mathématiques de Bourgogne (IMB) - UMR 5584, CNRS\\
	Université de Bourgogne, F-21000 Dijon, France}

\title{\Large \bf Maximal entropy random walks and central Markov chains\\
{\large A paradigm for  growth models and combinatorial identities}}

\begin{document}
	
	\date{}
	\maketitle
		
	\noindent
	\rule{\linewidth}{2pt}

	\noindent
	{\bf Abstract.} We introduce and develop the concept of Maximal Entropy Random Walks (MERWs) on Weighted Bratteli Diagrams (WBDs), maximizing entropy production along paths as a natural criterion for choosing random walks on networks. Initially defined for irreducible finite graphs, MERWs were recently extended to the infinite setting in  \cite{Duboux}. 	Bratteli Diagrams model various growth processes, such as the Young Lattice, where the Plancherel growth process emerges as a MERW. We show that MERWs are special cases of central Markov chains, which, in general, provide a powerful framework for deriving combinatorial identities. Regarding growing trees, in particular, we retrieve and extend Han’s hook-length formula for binary trees and demonstrate that the Binary Search Tree (BST) process is a MERW, recovering its asymptotic behavior. We also introduce preferential attachment to generalize BSTs.  For comb models, significant central measures appear, including the Chinese restaurant process, providing an alternative proof of the Poisson-Dirichlet limit distribution. Finally, we propose a Monte Carlo method, based on Knuth’s algorithm, to approximate MERWs. We apply it to a pyramidal growth model, drawing connections with the limit shape of Young diagrams under the Plancherel measure.

	\noindent
	\rule{\linewidth}{2pt}
	
	\vspace{10pt}
	
	
	\vspace{10pt}
		
	
	\noindent
\rule{\linewidth}{2pt}

\tableofcontents

	\vspace{10pt}

	\section{Introduction}

A commonly employed method to randomly explore a locally finite graph $G$, without relying on additional information, is to assume that a walker at any given node transitions uniformly at random to one of its neighboring nodes at each time step, independently of the past. 

This Markov process is referred to as a Generic Random Walk (GRW). Among all possible random walks, this choice maximizes entropy production at each step. More generally, for a Markov chain $(X_n)_{n\geq 0}$ on  $G$, it is natural to examine the asymptotic behavior of the entropy production of the first $n$-marginals
\begin{equation}
H((X_0, \cdots, X_n)) = -\sum_{x_0, \cdots, x_{n-1} \in G} p(x_0, \cdots, x_{n-1}) \ln(p(x_0, \cdots, x_{n-1})).
\end{equation}
Here, $p(x_0, \cdots, x_{n-1})$ denotes the probability distribution of $(X_0, \cdots, X_{n-1})$, which can be expressed as $\mu_0(x_0) p(x_0, x_1) \cdots p(x_{n-2}, x_{n-1})$, where $\mu_0(x_0)$ is the probability distribution of $X_0$, and $p(x, y)$ is the Markov kernel of the random walk. 

For an irreducible and positive recurrent Markov chain, this entropy production grows linearly, with a rate determined solely by $p$ (through its invariant probability $\pi$), given by

\begin{equation}\label{rate}
h(p) = \lim_{n \to \infty} \frac{H((X_0, \dots, X_{n-1}))}{n} = -\sum_{x, y \in G} \pi(x) p(x, y) \ln(p(x, y)).
\end{equation}
This quantity quantifies the entropy production per step under the stationary distribution.

Maximum Entropy Random Walks (MERWs) represent a paradigm shift from a local to a global perspective. These walks are designed to maximize entropy along their paths or, equivalently, the entropy rate (\ref{rate}). This approach was recently introduced in \cite{ref17,ref18,ref19} for finite irreducible graphs. Notably, the authors highlight the strong localization phenomenon exhibited by MERWs in slightly disordered environments. This characteristic has profound relevance in Quantum Mechanics, particularly in the context of the Anderson localization phenomenon (see \cite{konig2} for a mathematical overview).  The concept of MERWs is also intimately connected to Parry measures for subshifts of finite type, originally defined in \cite{Parry}. In addition, the case of infinite irreducible graphs has been investigated in \cite{Duboux}, revealing several phenomena that do not appear in the finite setting. As a matter of fact, since the right-hand side of (\ref{rate}) does not make sense when the Markov kernel is not positive recurrent, the appropriate definition of such walks remains somewhat unclear. Let us recall some definitions and properties of these walks.

\subsection{Irreducible framework}

\noindent
{\textit{i) The classical finite setting.}} When $G$ is an irreducible finite graph, the Perron-Frobenius theorem ensures the existence and uniqueness (up to a positive constant) of a positive right (resp. left) eigenvector $\psi$ (resp. $\phi$) of the $0/1$-adjacency matrix $A$ of the graph, associated with the spectral radius $\rho$. It can be easily shown that there is a unique random walk $(X_n)_{n\geq 0}$ on $G$ that maximizes the rate of entropy (\ref{rate}). 

This walk, called the Maximum Entropy Random Walk (MERW) on $G$, has a Markov kernel and an invariant probability measure given by
\begin{equation}\label{MERWexp}
p(x, y) = A(x, y) \frac{\psi(y)}{\rho \, \psi(x)} \quad \text{and} \quad \pi(x) = \phi(x) \psi(x).
\end{equation}
The eigenvectors $\psi$ and $\phi$ are normalized so that $\pi$ is a probability measure. 

Besides, one can easily show that the corresponding entropy rate is $h_{\rm MERW}=\ln(\rho)$. Interestingly, one can note that all trajectories of length $n$ between vertices $x$ and $y$ are equally probable, as shown by
\begin{equation}\label{unifipath}
\mathbb P(X_0=x,\cdots,X_n=y| X_0=x,X_n=y)=\frac{\psi(y)}{\rho^n\psi(x)}.    
\end{equation}
We shall prove (see Proposition \ref{uniform}) that this property characterizes MERWs on $R$-positive graphs (see below for a definition). 

To extend the scope, the adjacency matrix $A$ can be replaced with a weighted variant (strictly positive on edges), and the MERW can be chosen to maximize 
\begin{equation}\label{rate0}
h(p) = -\sum_{x, y \in G} \pi(x)p(x, y) \ln \left(\frac{p(x, y)}{A(x, y)}\right),
\end{equation}
over  positive-recurrent Markov kernels $p$ on $G$. When the entries $A(x,y)$ are non-negative integers, this formulation can be interpreted as a MERW on a multi-edge graph. Additional constraints, such as energy conditions, can be introduced as discussed in \cite{dixit}.

  Although there are only a limited number of solvable models where the spectral radius and the associated wave function are explicitly known, determining these in general remains a challenging task. For specific examples, such as (truncated) Cayley trees and ladder graphs, refer to \cite{Ochab}. For smaller graphs, it is feasible to compute these values numerically and conduct computer simulations of the MERW.\\

\noindent
{\textit{ii) The infinite setting.}}  The proper definition of a MERW on an infinite irreducible weighted graph $G$ has been addressed in \cite{Duboux}. This is primarily achieved using the theory of non-negative infinite matrices, as presented in \cite{VJI,VJII}. 

In this context, existence and uniqueness are no longer guaranteed. There are mainly two cases: the $R$-recurrent (resp. the $R$-transient) situation, characterized by
\begin{equation}\label{Rrec}
\sum_{n\geq 0} \frac{A^{n}(x,y)}{\rho^n}=\infty\quad  (\text{resp.} <\infty),
\end{equation}
for some (or equivalently all) $x,y$ in $G$. Here, $\rho$ denotes the inverse of the radius of convergence $R$ of the corresponding power series (which turns out not to depend on $x,y$). It is referred to as the combinatorial spectral radius in \cite{Duboux}, as it depends on the asymptotic weighted number of paths of length $n$ in the graph. 

Roughly speaking, a MERW on $G$ is still defined by the Markov kernel on the left-hand side of (\ref{MERWexp}), where $\psi$ is a positive eigenfunction associated with the weighted and infinite adjacency matrix $A$. In addition, it is elucidated in \cite{Duboux} how these random walks optimize the entropy rate.

The case where the graph is $R$-positive, meaning it is $R$-recurrent and $A^n(x, y)\rho^{-n}$ does not converge to zero, is very similar to the finite setting. In this scenario, there exists a unique MERW maximizing (\ref{rate0}) over positive-recurrent kernels, and the maximum is equal to $\ln(\rho)$. 

When the graph is $R$-recurrent but $A^n(x,y)\rho^{-n}$ tends to zero, existence and uniqueness are still maintained, but $\ln(\rho)$ is no longer a maximum of (\ref{rate0}), only a supremum. 

Besides, in the $R$-recurrent situation, the unique MERW is well approximated by considering any nested exhaustive sequence of irreducible finite subgraphs $\bigcup_{n\geq 0}\uparrow G_n = G $ and the corresponding sequence of classical  MERWs. 
 
The $R$-transient situation is much more complex. There exists a necessary and sufficient criterion for existence (see Theorem 2.1 in \cite{Duboux}), which is quite difficult to handle when the weighted graph is not locally finite, and uniqueness is no longer guaranteed.

 As a matter of fact, given a base point $o$ in $G$, and $\lambda \geq \rho$, the set
\begin{equation}
\mathcal C=\{\psi : G \rightarrow (0,\infty) : A\psi =\lambda \psi\text{ with } \psi(o)=1\},
\end{equation}
is a convex set whose extremal points can be described by the Martin boundary theory. Here again, $\ln(\rho)$ is the supremum of (\ref{rate0}) over positive-recurrent kernels. 

This explains why the case where $\rho$ in (\ref{MERWexp}) is replaced by $\lambda > \rho$ is not considered, even though the corresponding random walks maximize the pathwise entropy conditionally on their length and endpoints, as (\ref{unifipath}) highlights in the unweighted case. 

However, finite approximations of transient MERWs appear to be more enigmatic. An example in Section 2.3 of \cite{Duboux} is provided, where all finite approximations lead to a quantized subset of all the transient MERWs as they were previously defined, raising questions about the appropriate  definition of MERWs in the irreducible infinite framework.

\subsection{Beyond the irreducible framework : Bratteli Diagrams.}

 Among the wide variety of such networks, one notable example is Directed Acyclic Graphs (DAGs), with rooted trees being a particularly simple case. Trees have the advantage of possessing rigid hierarchical levels. Between these two models, we have chosen to investigate the case of Bratteli Diagrams (BDs). We refer to Definition \ref{defbratteli}  or Figure \ref{figbd} for an example.  

\begin{figure}[H]
	\centering
	\includegraphics[scale=.8]{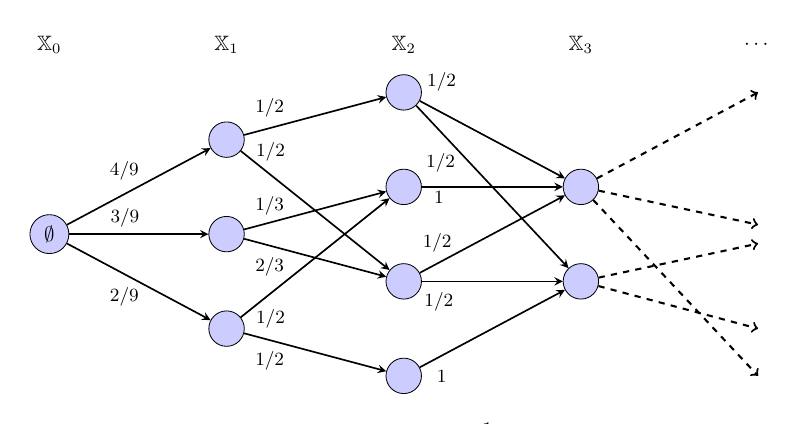}
	\caption{An example of Brattelli Diagram with some transition probabilities}
	\label{figbd}
\end{figure}

These diagrams are associated with rich algebraic, combinatorial, and probabilistic structures. A gentle introduction can be found in \cite{kerov}, and the Young or Pascal lattice can be mentioned as famous examples of BDs. Moreover, upon closer inspection, these graphs encode a large class of growth models.\\

\noindent
{\textit{i) About the definition}.} To begin with, the proper definition of what constitutes a MERW on a non-irreducible graph is not entirely clear. 

We have chosen to define MERWs by truncation and approximation. The definition we provide in (\ref{def}) is, in some sense, consistent with the general one given in \cite{Duboux} (see Section \ref{section:connexion}), but it is not equivalent, particularly in the $R$-transient case. 

All trajectories of length $n$ between vertices $x$ and $y$ still have the same probability or, more generally, a probability proportional to their weight (see Proposition \ref{uniform} and Corollary \ref{coro}). 

However, this property no longer characterizes MERWs on BD as defined in this paper, but instead characterizes central Markov chains on WBDs, as defined in \cite{kerov}.  

Here, we slightly generalize the definition of a central Markov chain and the related results to gain flexibility, allowing them to not have full support. When restricted to the latter, they yield a classical central Markov chain. Let us briefly introduce this concept and the main result.\\

\noindent
{\textit{ii) Central Markov chains.}} In some sense, we shall see that we might refer to central Markov chains as Weak-MERWs. 

For instance, on the Pascal lattice, the Polya urn process with two colors is a central Markov chain but not a MERW (see Remark \ref{rem:polya}). 

Similar to (\ref{MERWexp}), central Markov chains can be characterized by the so-called Positive Harmonic Function (PHF) $\varphi$ satisfying
\begin{equation}
\varphi(x) = \sum_{x \nearrow y} w(x, y)\varphi(y),
\end{equation}
where $x \nearrow y$ denotes that $(x, y)$ is an edge of the BD, and $w(x, y) > 0$ is some weight.
 The corresponding central Markov kernel is then given by $p(x,y)=w(x,y)\varphi(y)/\varphi(x)$.

\begin{rem}
Caution, this does not imply that the combinatorial spectral radius $\rho$, if it exists, is equal to one here. In fact, a vertex $x$ of a BD encodes in itself the distance $n$ to the root, and thus, in a certain way, $\varphi$ can be seen as a space-time harmonic function. 

We refer to Section \ref{section:connexion} for more details and to the Appendix \ref{appendix} for (slightly adapted) classical results about central Markov chains.
\end{rem}

\noindent
{\textit{iii) A powerful combinatorial identity}.} Moreover, since all paths starting from the root and ending at the same point have a probability proportional to their weight, this leads to the powerful and general combinatorial identity (see (\ref{combipower}) for a more precise statement):
\begin{equation}\label{combipower0}
\sum_{y \in \mathbb{X}_n} \#\{\varnothing \to y\} \varphi(y) = 1.
\end{equation}
Here, $\#\{\varnothing \to y\}$ denotes the weighted number of paths from the root $\varnothing$ of the Bratteli Diagram  to a vertex $y \in \mathbb{X}_n$ at a distance $n$ of the root. 

We shall see that this allows proving some combinatorial identities or reinterpreting them by unconventional means. 

For instance, the well-known binomial expansion and the Plancherel identity for the number of Young diagrams fall into this category. In particular, we give a proof of  Han's hook length formula  for binary trees \cite{Han1} using this method and we generalize it in  Corollaries \ref{Han-d} and \ref{comb4}. We also recover the well-known identity (\ref{stirling}) for Stirling numbers of the second kind.\\

\noindent
{\textit{iii) Tree growth process}.}
Most of these identities arise from the study of BDs associated with tree-growing models in Section \ref{section:tree} (see in particular  Theorem \ref{bratstructure}). 

We focus on rooted trees because they are equipped with a well-known hook-length formula for increasing labeling, similar to that for Young tableaux, which allows us to enumerate paths. 

Notably, we show that the Plancherel growth process is the unique MERW on the Young lattice (see Proposition \ref{plancherelgrowth}). 

We also describe central Markov chains and MERWs for various models of growing trees. In particular, we reinterpret the well-known Binary Search Tree (BST) process as a MERW (see Proposition \ref{BST}), allowing us to recover in an elegant way the well-known asymptotics of this process (Corollary \ref{coro:BST}).  

Some of these models can also be interpreted as aggregation processes. For instance, we show that the Chinese restaurant process is a central Markov chain (but not a MERW). Again, this approach enables us to recover, in a manner that appears both direct and simple, some well-known asymptotics of these processes (Corollary \ref{chinese}).\\

\noindent
{\textit{iv) Computer simulations}.}
However, in practice, it is impossible in most cases to explicitly compute central Markov chains or MERWs on a Bratteli Diagram. 

In Section \ref{sec:algo}, we propose a Monte Carlo method to estimate the asymptotic number of paths. This method is based on a Monte Carlo algorithm by Knuth \cite{knuth97}, which allows for estimating the number of leaves in a finite tree. 

We apply our algorithm to a growth model of pyramidal diagrams (a generalization of the Kreweras random walk), which can be interpreted as two-dimensional Young diagrams.\\

\noindent
{\bf Outline of the paper.} In Section \ref{sec:1}, we introduce MERWs on Weighted Bratteli Diagrams (WBDs) and characterize them using central Markov chains. We establish connections with previous works, investigate combinatorial identities and demonstrate that the Plancherel growth process is the unique MERW on the Young Lattice. 

Section \ref{section:tree} delves into prefix trees, leveraging a hook-length formula to compute path asymptotics and describe central measures. Ergodic measures, in particular, are linked to fragmentation processes on the infinite genealogical tree. Theorem \ref{bratstructure} characterizes all central Markov chains in the unweighted case.

In Sections \ref{sec:examples} and \ref{sec:combagreg}, we explore specific models (both weighted and unweighted) involving $d$-ary trees, infinite combs, and aggregation processes.
Notably, we extend Han's hook length formula for binary trees and compute the unique MERW of a generalized BST process with preferential attachment. The main result here is Theorem \ref{thm:BSTgen}. We connect comb models to the Chinese restaurant process, and analyze their Poisson-Dirichlet asymptotics. Also, we introduce constraints that lead to ties with the Young lattice and Kreweras random walks. 

Section \ref{sec:algo} presents a Monte Carlo algorithm, inspired by Knuth’s method, for approximating MERWs when direct computation of transition probabilities is impractical. We apply this algorithm to some pyramidal model and conjecture its limit shape, drawing parallels with the Plancherel growth process.

Finally, Appendix \ref{appendix} presents the main known results about (our slightly generalized version) of central Markov chains, ensuring readability and completeness.

\section{MERWs on WBDs}

\label{sec:1}

\setcounter{equation}{0}

We begin with an introduction to WBDs, modifying the usual framework detailed in \cite{kerov} to incorporate \textbf{countably infinite} level sets. Subsequently, we introduce random walks and  define MERWs on these lattices.

\subsection{WBDs}

\label{sec:WBD}

\begin{defi}\label{defbratteli}
	A graded graph $\mathbb{X}=\bigsqcup_{n\geq 0} \mathbb{X}_n$ is called a Bratteli Diagram (BD) if it satisfies the following properties.
	\begin{enumerate}
		\item[i)] For every edge $(x,y)$, we have $x \in \mathbb{X}_n$ and $y \in \mathbb{X}_{n+1}$ for some $n \geq 0$.
		\item[ii)] There exists a unique vertex $\varnothing$, called the root, with no incoming edges.
		\item[iii)] Every vertex has at least one outgoing edge.
		\item[iv)] Each level set $\mathbb{X}_n$ is finite or  {\textbf{countably infinite}}.
	\end{enumerate}     
\end{defi}

\noindent
{\textit{i) Notations.}}		
Given $x \in \mathbb{X}$, we denote by $n_x$ the rank $k$ of the level set $\mathbb{X}_k$ that contains $x$. We write $x \nearrow y$ to indicate that $(x,y)$ is an edge in the BD. A path $s$ in the BD is denoted by a sequence $s_0 \nearrow s_1 \nearrow \cdots$, or by concatenation $s_0 s_1 \cdots$. 

The set of all infinite paths starting from $\varnothing$ is denoted by $\mathcal{T}$, and the set of all finite paths of length $n$ starting from the root is denoted by $\mathcal{T}_n$. By convention, we set $\mathcal{T}_{-1} = \emptyset$. Also, we denote by $\{x \to y\}$ the set of all finite paths that start from $x$ and end at $y$. Note that each of them has the same length equal to $n_y-n_x$.\\

\noindent
{\textit{ii) Weighted structure and combinatorial dimension.}}
 To each edge $x \nearrow y$, one can assign a positive weight $w(x,y)$, and by extension, to every finite path $s = s_0 \nearrow \cdots \nearrow s_n$ of length $n$, one can assign the weight
\begin{equation}
w_s = \prod_{i=0}^{n-1} w(s_i,s_{i+1}).
\end{equation}

We refer to $(\mathbb{X}, w)$ as a Weighted Bratteli Diagram (WBD). In the absence of a specified weight function, we will implicitly assume that $w \equiv 1$. Obviously, one may extend $w$ to $\mathbb{X} \times \mathbb{X}$ by setting $w(x,y) = 0$ whenever $(x,y)$ is not an edge.

\begin{defi}\label{def:combidim}
 The combinatorial dimension between vertices $x$ and $y$ is defined as follows: for $x=y$, we set $d(x,y)=1$, and for $x \neq y$, we set
\begin{equation}
d(x,y) = \sum_{s \in \{x \to y\}} w_s.
\end{equation}	
\end{defi}

Note that $d(x,y) = \text{card}(\{x \to y\})$ when $x \neq y$ and $w \equiv 1$. To generalize, we extend $d$ by setting $d(x,F) = \sum_{y \in F} d(x,y)$ for any vertex $x \in \mathbb{X}$ and subset $F \subseteq \mathbb{X}$. The following hypothesis will be maintained throughout the paper.
\begin{hyp}\label{Ass2}
For every $n \in \mathbb{N}$, we have $\displaystyle d(\varnothing, \mathbb{X}_n) = \sum_{s \in \mathcal{T}_n} w_s < \infty.$
\end{hyp}

\subsection{Random Walks and Entropy maximization}

{\textit{i) Random walks setting.}} A random walk  $(X_n)_{n \geq 0}$ is defined by a Markov kernel $p(x,y)$, $x, y \in \mathbb{X}$, which is equal to zero when $(x,y)$ is not an edge of the graph.  

Let us introduce, for any $n\geq 0$ and  $s\in \mathcal{T}_n$, 
\begin{equation}\label{cylinder}
C_s = \{ t \in \mathcal{T} : \forall\, 0 \leq k \leq n,\, t_k = s_k \}.
\end{equation}
We shall denote by $\mathcal{A}$ the $\sigma$-algebra on $\mathcal{T}$ generated by all these cylinder sets. 

A random walk is characterized by its distribution $\mu$ starting from the root. This is a probability measure on the measurable space $(\mathcal{T}, \mathcal{A})$. Its marginal distribution on $\mathcal{T}_n$ is given for all $s \in \mathcal{T}_n$ by
\begin{equation}
\mu(C_s) = p(\varnothing, s_1) \cdots p(s_{n-1}, s_n).
\end{equation}
For a finite path $s = s_0 \nearrow \cdots \nearrow s_n$ of length $n$, we shall set $p(s) = p(s_0, s_1) \cdots p(s_{n-1}, s_n)$ in such a way that 
\begin{equation}
\mathbb P(X_1=s_1,\cdots,X_{n}=s_n|X_0=s_0)=p(s).
\end{equation}

We denote by $\mathcal{RW}$ the set of probability distributions on $(\mathcal{T}, \mathcal{A})$ which come from  a random walk and by $\mathcal{RW}_n$ their restrictions to $\mathcal{T}_n$. Note that one can embed $\mathcal{RW}_n \hookrightarrow \mathcal{RW}$.

\begin{defi}\label{supportdef}
	The support of a random walk $(X_n)_{n\geq 0}$, or equivalently the support of the corresponding probability measure $\mu\in\mathcal{RW}$, is defined by 
	\begin{equation}\label{support0}
	\mathbb{S} = \{x \in \mathbb{X} : \exists n\geq 0,\;\mathbb{P}_\varnothing(X_n = x) > 0\}.
	\end{equation}
\end{defi}

\noindent
{\textit{ii) Entropy maximization.}} For any non-negative integer $n$ and any probability distribution $\nu$ on $\mathcal{T}_n$, one can define 
\begin{equation}\label{weightentropy1}
H_w(\nu) = H(\nu) + \sum_{s \in \mathcal{T}_n} \ln(w_s) \nu_s,
\end{equation}
where $H(\nu) = -\sum_{s \in \mathcal{T}_n} \nu_s \ln(\nu_s)$ is the usual Shannon entropy of $\nu$. 

One can easily check that $H_w(\nu)$ can be expressed as
\begin{equation}\label{weightentropy2}
H_w(\nu) = \ln\left(\sum_{s \in \mathcal{T}_n} w_s \right) - D_{\rm KL}(\nu \| \mu_n),
\end{equation}
with
\begin{equation}\label{mun}
\mu_n(s) = \frac{w_s}{\sum_{t \in \mathcal{T}_n} w_t}\quad \text{and} \quad D_{\rm KL}(\nu \| \mu_n) = \sum_{s \in \mathcal{T}_n} \nu_s \ln\left( \frac{\nu_s}{\mu_n(s)} \right).
\end{equation}

Recall that $D_{\rm KL}$ is the Kullback–Leibler Divergence (KLD) or the relative entropy (see \cite{KLD}). It follows from Assumption \ref{Ass2} and the properties of the KLD that $\nu\longmapsto  H_w(\nu)$ is well-defined and achieves its maximum value, given by $\ln\left(\sum_{s \in \mathcal{T}_n} w_s\right)$, for $\nu = \mu_n$. 

\begin{rem}\label{mun0}
The distribution $\mu_n$ is the unique probability measure on $\mathcal T_n$ such that the probability of any path $s \in \mathcal T_n$ is proportional to its weight $w_s$. 

In particular, when $w \equiv 1$, it corresponds to the uniform probability distribution  and thus  the entropy is equal to $\ln(\# \mathcal T_n)$.
\end{rem}

The proof of the following lemma is straightforward. The key point is that $\mu_n$ represents the distribution of a random walk (a priori depending on $n$) restricted to $\mathcal{T}_n$.

\begin{lem}\label{lem}
	For any non-negative integer $n$, the distribution $\mu_n$ belongs to $\mathcal{RW}_n$. More precisely, the corresponding transition kernel is given for all $x \nearrow y$ by
	\begin{equation}\label{finiteMERW}
	p_n(x,y) = w(x,y) \frac{d(y, \mathbb{X}_n)}{d(x, \mathbb{X}_n)} = \frac{w(x,y) d(y, \mathbb{X}_n)}{\sum_{x \nearrow y'} w(x, y') d(y', \mathbb{X}_n)}.
	\end{equation}

	Furthermore, for all $s \in \mathcal{T}_n$, one can write 
	\begin{equation}\label{phin}
	p_n(s) = w_s \frac{\varphi_n(y)}{\varphi_n(x)}, \quad \text{with} \quad \varphi_n(z) = \frac{d(z, \mathbb{X}_n)}{d(\varnothing, \mathbb{X}_n)}.
	\end{equation}

	Additionally, one has $\varphi_n(\varnothing) = 1$ and for all $x$ with $|x| < n$,
	\begin{equation}\label{harmonicn}
	\varphi_n(x) = \sum_{x \nearrow y} w(x, y) \varphi_n(y).
	\end{equation}
\end{lem}

\begin{rem}
	To perform the random walk corresponding to $\mu_n$, we need to enumerate, at each point $x$, the weighted number of paths $d(x, \mathbb{X}_n)$ leading to the $n$th level set $\mathbb{X}_n$. Besides, note that the restriction of $\mu_n$ to $\mathcal{T}_k$ for $k < n$ is not equal to $\mu_k$ in general. 
	
	For instance, in Figure \ref{figbd}, the transition probabilities correspond to the unweighted case $w \equiv 1$ and $n=4$, but  even if each of the $9$ paths of length $3$ starting from the root has the same probability, equal to $1/9$, the distributions of the trajectories of lengths $k=1,2$ are not uniform.
\end{rem}

\subsection{A consistent definition} 

\label{sec:def}

Since $\mathbb{X}$ is countably infinite and all the sequences $(\varphi_n(x))_{n \geq 1}$ for $x \in \mathbb{X}$ are bounded by $1$, we get from Cantor's diagonal argument that there exists a subsequence $(n_k)_{k \geq 1}$ such that $\varphi_{n_k}$ converges pointwise to some $\varphi^\ast : \mathbb{X} \longrightarrow [0, \infty[$ as $k$ tends to infinity. 

Furthermore, by using Lemma \ref{lem} (specifically (\ref{harmonicn})) and applying the dominated convergence theorem (with the help of Assumption \ref{Ass2}), one can verify that $\varphi^\ast$ is a Non-Negative Harmonic Function (NNHF) as defined below.

\begin{defi}\label{defsat}
	A function $\varphi : \mathbb{X} \longrightarrow [0, \infty)$ is said to be a Non-Negative Harmonic Function (NNHF)  if $\varphi(\varnothing) = 1$ and, for all $x \in \mathbb{X}$,
	\begin{equation}\label{harmonicc}
	\varphi(x) = \sum_{x \nearrow y} w(x,y) \varphi(y).
	\end{equation} 
\end{defi}

 Equivalently, there exists a Markov kernel $p^\ast$ on $\mathbb X$ such that $p_{n_k}$ converge pointwise to $p^\ast$. Still equivalently,  the sequence of probability distribution  $\mu_{n_k} \in \mathcal{RW}_{n_k}$ converges in law to some $\mu^\ast \in \mathcal{RW}$, in the sense that for all cylinder sets $C_s \in \mathcal{A}$ with $s \in \mathcal{T}_n$ and $n \geq 0$, one has \begin{equation}\label{cvgdistrib} \lim_{k \to \infty} \mu_{n_k}(C_s) = \mu^\ast(C_s). \end{equation}
Note that (\ref{cvgdistrib}) does not depend on how we embed each $\mathcal{RW}_{n_k}$ into $\mathcal{RW}$. 

Moreover, these three characteristics $(\mu^\ast, p^\ast, \varphi^\ast)$ are related by the following formulas: for any path $s$ from the root to some $y \in \mathbb{X}$,
\begin{equation}\label{link}
\mu^\ast(C_s) = p^\ast(s) = w_s\, \varphi^\ast(y). 
\end{equation}
It suffices indeed to look at Lemma \ref{lem} again. This leads to the following definition.

\begin{defi}\label{def}
	A MERW on a WBD is any random walk associated with any limit point $\mu^\ast$ (equivalently  $p^\ast$ or $\varphi^\ast$) as defined above. 
\end{defi}

\begin{defi}\label{defbrattelisat}
	A Bratteli diagram $\mathbb{Y}$ is a Saturated Sub-Bratteli Diagram (SSBD) of $\mathbb{X}$ if 
	\begin{enumerate}
		\item[a)] $\mathbb{Y} \subset \mathbb{X}$ as a subgraph.
		\item[b)] For every $y \in \mathbb{Y}$ and $x \in \mathbb{X}$ such that $x \nearrow y$ in $\mathbb{X}$, one has $x \in \mathbb{Y}$.
	\end{enumerate}
\end{defi}

\begin{prop}\label{suppssbd}
The support $\mathbb X^\ast$ of a MERW,  corresponding to the NNHF $\varphi^\ast$, is a SSBD which  is given by
\begin{equation}\label{support}
\mathbb{X}^\ast =  \{x \in \mathbb{X} : \varphi^\ast(x) > 0\}.
\end{equation}
\end{prop}

\begin{proof}
Note that (\ref{support}) comes easily from (\ref{link}). Let us prove that $\mathbb X^*$ is a SSBD. 

To this end, assume that $\varphi^\ast(y)>0$ and let $x\in\mathbb X$ such that $x\nearrow y$.  By using (\ref{harmonicc}) we obtain that $\varphi^\ast(x)\geq w(x,y)\varphi^\ast(y)>0$, which completes the proof. 
\end{proof}

\subsection{MERWs as particular central Markov chains.} 

\label{sec:central}

The subsequent results arise from the bijection between central measures and PHFs, as well as the construction of the associated Martin boundary detailed in \cite{kerov}. 

Since this article allows {\bf countably infinite} level sets,  we have concisely reviewed and adapted their results in Appendix \ref{appendix} for the sake of completeness. The notions of central measure and central Markov chain are replaced by the notions of saturated central measure and saturated Markov chain (see Definition \ref{defmerw}), and PHFs are replaced by NNHFs. 

Roughly speaking, the main difference is that we allow the common support of these objects to be a SSBD, not necessarily the whole BD.

To understand Corollary \ref{coro} below, we need to briefly recall some results and introduce some additional notations. 

First,  a central measure on a WBD $\mathbb{X}$ is a probability measure $\mu$ on $(\mathcal{T}, \mathcal{A})$ such that $w_s \mu(C_t) = w_t \mu(C_s)$ for all $y\in\mathbb X$ and $s, t \in\{\varnothing\to y\}$. 

As a consequence, it turns out that $\mu \in \mathcal{RW}$ and the corresponding random walk has full support. Besides, conditionally on any starting and ending points $x, y$, any trajectory $s$ from $x$ to $y$ has a probability proportional to its weight $w_s$ (equal to $1/\#\{x \to y\}$ when $w \equiv 1$). 

Furthermore, there are one-to-one correspondences between central measures and PHFs, as well as between ergodic central measures and the (combinatorial) Martin boundary. 

Here, we shall say that a path $t \in \mathcal{T}$ converges to $\zeta \in \partial \mathbb{X}$, a point of the Martin boundary, if and only if, for all $x \in \mathbb{X}$,
\begin{equation}\label{Martin}
K(x, \zeta) := \lim_{n \to \infty} \frac{d(x, t_n)}{d(\varnothing, t_n)} \quad \text{exists}.
\end{equation}

Moreover, we denote by $m_n$ the $n$th marginal of $\mu_n$, defined for any $y \in \mathbb{X}_n$ by
\begin{equation}\label{marginalunif}
m_n(y) = \sum_{s \in \{\varnothing \to y\}} \mu_n(s) = \frac{d(\varnothing, y)}{d(\varnothing, \mathbb{X}_n)}.
\end{equation}

\begin{coro}\label{coro}
	Let $(X_n)_{n \geq 0}$ be a MERW and let $\mu^\ast$, $\varphi^\ast$ and $\mathbb{X}^\ast$ be defined as in Section \ref{sec:def}. Then $\mu^\ast$ is a central measure on the SSBD $\mathbb{X}^\ast$ and $(X_n)_{n \geq 0}$ is the associated central Markov chain. In particular, for all $x, y \in \mathbb{X}^\ast$ (defined in (\ref{support})) and all $t \in \{x \to y\}$ of length $n$:
	\begin{equation}\label{Unif}
	\mathbb{P}(X_1 = t_1, \ldots, X_{n-1} = t_{n-1} \mid X_0 = x, X_n = y) = \frac{w_t}{\sum_{s \in \{x \to y\}} w_s}.
	\end{equation} 
	
	Furthermore, there is a bijective relationship between MERWs and the limit points $m^\ast$, in the space of probability measures on $\mathbb{X} \sqcup \partial \mathbb{X}$, of the sequence $(m_n)_{n \geq 0}$. Each MERW is characterized by a NNHF $\varphi^\ast$ such that 
	\begin{equation}\label{eq:representationint}
	\varphi^\ast(x) = \int_{\partial \mathbb{X}} K(x, \zeta) \, m^\ast(d\zeta).
	\end{equation}    
\end{coro}

\begin{proof}
These results are mainly direct consequences of Appendix \ref{appendix} and Section \ref{sec:def}. 

Regarding the last point, we recall that $\mathbb{X} \sqcup \partial \mathbb{X}$ is a compact metric space. Besides, one has 
\begin{equation}\label{eq:representationint_n}
\varphi_n(x)=\frac{d(x, \mathbb{X}_n)}{d(\varnothing, \mathbb{X}_n)} = \int_{\mathbb{X}_n} \frac{d(x, y)}{d(\varnothing, y)} m_n(dy),
\end{equation}

If $m^\ast$ is the limit point of $(m_{n_k})_{k\geq 0}$, it is supported on $\partial \mathbb X$ since $m_n(\mathbb X_n)=1$ for all $n\geq 0$. We get from \eqref{eq:representationint_n} we get that $(\varphi_{n_k})_{k\geq 0}$ converges to some $\varphi^\ast$ since $y\to \frac{d(x, y)}{d(\varnothing, y)}$ is bounded and continuous on $\mathbb X\sqcup \partial \mathbb X$ for all $x\in\mathbb X$. Finally,  we obtain \eqref{eq:representationint}.

Reciprocally, if $(\varphi_{n_k})_{k\geq 0}$ converges to some $\varphi^\ast$, then by compactness, there exist some subsequence $(l_k)_{k\geq 0}$ of $(n_{k})_{k\geq 0}$ and $m^\ast$ a probability measure on $\partial \mathbb X$ such that $(m_{l_k})_{k\geq 0}$ converges to $m^\ast$. Again, we obtain  \eqref{eq:representationint}.
	
\end{proof}

\begin{rem}
Equation (\ref{Unif}) means that, conditionally on their starting and ending points, the trajectories of a MERW on a WBD have probabilities proportional to their weights. We shall prove that the converse is true for MERWs on positive $R$-recurrent graphs (see Proposition \ref{uniform}).
\end{rem}

\subsection{Connection with  MERWs on irreducible graphs} \label{section:connexion}

In this section we explore the {relationship} between MERWs on BDs as defined above and MERWs on irreducible weighted graphs $G$, possibly infinite, as introduced in \cite{Duboux}. 

Again, denote by $A$ the corresponding weighted adjacency matrix and fix a base point $o \in G$. Define the level sets $(\mathbb{X}_n)_{n \geq 0}$ recursively as follows: $\mathbb{X}_0 = \{\varnothing\} = \{(0, o)\}$ and
\begin{equation}
\mathbb{X}_{n+1} := \{(n+1, y) : \exists x \in G,\, (n, x) \in \mathbb{X}_n \text{ and } A(x, y) > 0\}.
\end{equation}

We say that $(n, x) \nearrow (n+1, y)$ {if and only if} $A(x, y) > 0$ and we define the weight function by $w((n, x); (n+1, y)) = A(x, y)$. Note that Assumption \ref{Ass2} is satisfied if 
\begin{equation}
\sup_{x \in G} \sum_{y \in G} A(x, y) < \infty.
\end{equation}

Let $R$ denote the convergence parameter as introduced in \cite{VJI} and $\rho = R^{-1}$ the so-called combinatorial spectral radius. For all $0 < r \leq R$ and all positive $(1/r)$-harmonic functions $\psi$, when they exist, one can produce a PHF $\varphi$ on $(\mathbb{X}, w)$ by setting 
\begin{equation}\label{product}
\varphi(n, x) = r^n \psi(x).
\end{equation}

Conversely, a uniform aperiodicity criterion is given in \cite{KestenRatio, spacemartin} to ensure that a PHF $\varphi$ can be written {in the manner described above}. Note that any MERW on the WBD can be {calculated} by studying the limit points, as $N$ goes to infinity, of
\begin{equation}\label{limitpoint}
\varphi_N(n, x) = \frac{d((n, x); \mathbb{X}_N)}{d(\varnothing; \mathbb{X}_N)} = \frac{A^{N-n} \mathds{1}(x)}{A^N \mathds{1}(o)}.
\end{equation}

When $G$ is finite, the two definitions of MERWs coincide since in that case the spectral theorem applies and we have 
\begin{equation}
\lim_{N \to \infty} \varphi_N(n, x) = \frac{\psi(x)}{\rho^n},
\end{equation}
where $\psi$ is the unique positive solution of $A \psi = \rho \psi$ with $\psi(o) = 1$. More generally, we get from \cite[Theorem 7.2]{VJI} the following result.
\begin{prop}
	Assume that $A$ is $R$-positive. Then the definition of the MERW on $G$ given in \cite{Duboux} is equivalent to the definition of the MERW stated in this paper.
\end{prop}

To go further, one can see that (\ref{Unif}) characterizes MERWs on  $R$-positive graphs. 

\begin{prop}\label{uniform}
	Assume that $A$ is $R$-positive. Then the unique MERW $(X_n)_{n \geq 0}$ is the unique random walk on $G$ such that, for all $n \geq 0$ and $x_0, \ldots, x_n \in G$,
	\begin{equation}\label{unifweight}
	\mathbb{P}(X_1 = x_1, \ldots, X_{n-1} = x_{n-1} \mid X_0 = x_0, X_n = x_n) = \frac{\prod_{i=0}^{n-1} A(x_i, x_{i+1})}{A^n(x_0, x_n)}.
	\end{equation}    
\end{prop}

\begin{proof}
	Assuming (\ref{unifweight}) and letting $\mathcal P_{n-1}^{y,z}$ be the set of paths from $y$ to $z$ of length $n-1$, one can easily check that 
	\begin{equation}\label{ratio}
	p(x, y) = \sum_{z\in G}\sum_{c\in\mathcal P_n^{y,z}} p(x\nearrow y\nearrow \cdots\nearrow z)= A(x, y) \sum_{z \in G} \frac{A^{n-1}(y, z)}{A^n(x, z)} \mathbb{P}_x(X_n = z).
	\end{equation}    
	
	Let us denote by $\varphi$ and $\psi$ the unique (up to a multiplicative term) positive left and right eigenvectors of $A$ associated with the combinatorial spectral radius $\rho$ such that $\langle \psi, \varphi \rangle = 1$ for the usual scalar product. 
	
	Let $d \geq 1$ be the period of $A$, \emph{i.e.\@} $d = {\rm gcd}(S)$ where $S = \{n \geq 1 : p^n(e, e) > 0\}$ for some $e \in G$. By using again \cite[Theorem 7.2]{VJI}, there exists a unique $0 \leq k < d$ such that $A^{dn+k}(x, y) > 0$ for all $n$ sufficiently large and  
	\begin{equation}\label{asymp0}
	A^{dn+k}(x, y) \underset{n \to \infty}{\sim} \rho^{dn+k} \psi(x) \varphi(y).
	\end{equation}

	Therefore, we deduce the result from (\ref{ratio}) and the dominated convergence theorem.
\end{proof}

\begin{rem}
	The latter proof provides a method to approximate MERW on $R$-positive graphs. 
	
	Indeed, let $f : G \longrightarrow (0, \infty)$ satisfy $\sum_{y \in G} f(y) \varphi(y) < \infty$, where $\varphi$ is the left $\rho$-eigenfunction. Set for any $n \geq 1$ and $x, y \in G$,
	\begin{equation}
	p_{n}(x, y) = A(x, y) \frac{A^{n-1} f(y)}{A^{n} f(x)}.
	\end{equation}
	
	In the case when $f \equiv 1$, the probability $p_n(x, y)$ is simply the (weighted) proportion, among the paths of length $n$ starting from $x$, of those beginning with the transition $x \nearrow y$. When $n = 1$, this corresponds to the usual GRW. 
	
	Similar arguments as before show that $p_n$ converges pointwise to the MERW kernel as $n$ goes to infinity.
\end{rem}

\subsection{Illustrative examples}

\label{sec:illustra}

This section is devoted to some examples that illustrate the concepts introduced in the previous sections (but also in Appendix \ref{appendix}) and present some techniques and approaches for studying MERWs on WBDs.\\

\noindent
{\textit{i) About the support and uniqueness.}} The first two toy models of BD $\mathbb{X}$ we consider are detailed in Figure \ref{toy}. The first one shows that there may not exist MERWs with full support, that is with $\mathbb X^*=\mathbb X$,  whereas the second one highlights that there may exist several MERWs.\\

\textbullet\; The first model consists of level sets $\mathbb{X}_n = \{n_l, n_r\}$, for $n \geq 1$, with the edges $\varnothing\nearrow 1_i$, $i\in\{l,r\}$,  $n_l\nearrow (n+1)_l$, $n_r\nearrow (n+1)_r$, $n\geq 1$. We find that $d(\varnothing, n_l) = 1$, $d(\varnothing, n_r) = n$, and more generally $d(n_l, m_l) = d(n_r, m_r) = m - n$, $d(n_l, m_r) = m - n$, and $d(n_r, m_l) = 0$ for all $m \geq n \geq 1$. 

By using (\ref{Martin}) one can see that $\partial \mathbb{X}$ contains only one point, denoted as $\delta$. The unique  MERW corresponds to the unique extremal NNHF $\varphi^\ast = \varphi_\delta$, where $\varphi_\delta(n_l) = 1$ and $\varphi_\delta(n_r) = 0$. This implies that $\mathbb{X}^\ast = \{\varnothing, 1_r, 2_r, \ldots\}$, indicating that no MERW with full support exists.\\

\textbullet\; Regarding the second model, we denote $\mathbb{X}_n$ as $\mathbb{L}_n \sqcup \mathbb{R}_n$, representing the left and right-hand sides when $n \geq 1$. One can check that $\partial \mathbb{X} = \{l, r\}$, where $\varphi_{l}(x) = 1 / {\rm card}(\mathbb{L}_n)$ if $x \in \mathbb{L}_n$ and $\varphi_{l}(x) = 0$ if $x \in \mathbb{R}_n$. Similarly, $\varphi_{r}$ is defined by exchanging the letters $l$ and $r$. 

Considering $m_n$ as defined in (\ref{marginalunif}), we observe that $m_n(\mathbb{L}_n) = 1/2$ (resp. $1/3$, $2/3$) when $n = 3k + 1$ (resp. $n = 3k + 2$, $n = 3k + 3$) for some $k \geq 0$.  In the light of Corollary \ref{coro}, there are three distinct limit points $m^\ast$, each corresponding to one of the three MERWs depicted in Figure \ref{toy}. Therefore, there is no uniqueness for this model.

\begin{figure}[H]
	\centering
	\includegraphics[width=0.75\textwidth]{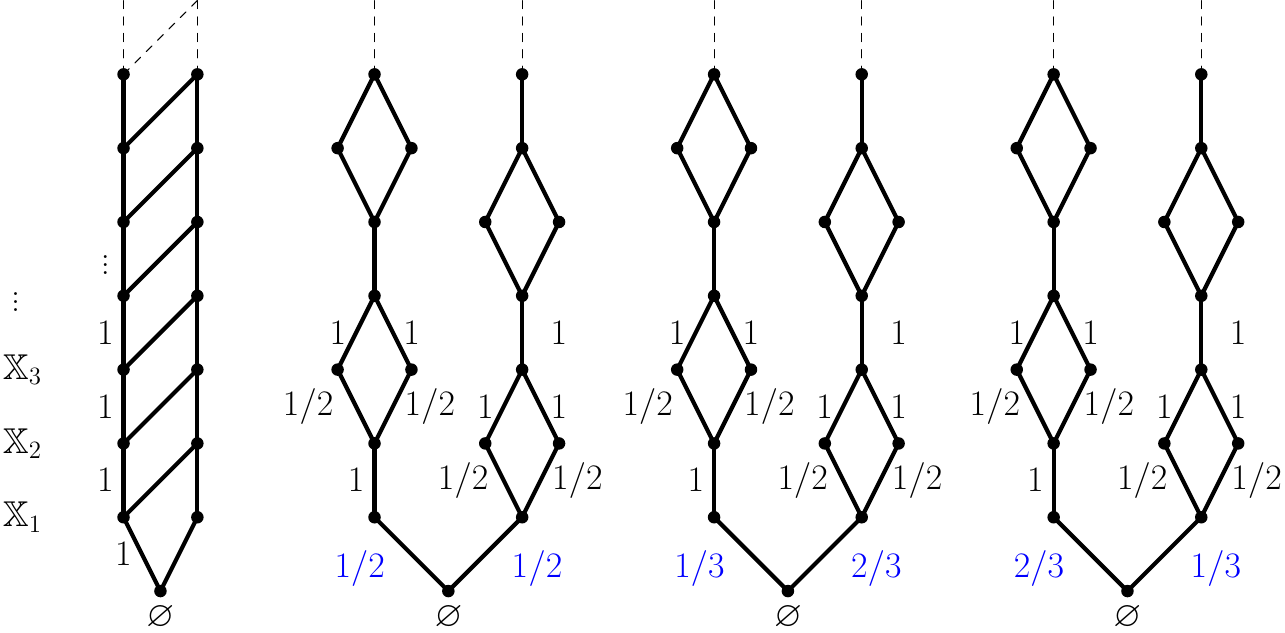}
	\captionsetup{width=0.8\textwidth}
	\caption{The first example, on the left, show that the unique MERW has no full support. On the the second one, it is represented  the transition probabilities of the three possible MERWs.}	
	\label{toy}
\end{figure}

\noindent
{\textit{ii) Correspondence with irreducible graphs.}} Consider the usual BD $\mathbb{X} = \mathbb{N}^2$, where the $n$th level set is $\mathbb{X}_n = \{(k, n-k) : 0 \leq k \leq n\}$. This BD is called the Pascal lattice in \cite{kerov}.

It is known that there is a one-to-one correspondence between NNHFs and probability distributions $m$ on the interval $[0,1]$. This correspondence is established through 
\begin{equation}
\varphi_m(k, n-k) = \int_0^1 p^k (1-p)^{n-k} m(dp).
\end{equation}

It turns out that an infinite path $t_n = (k_n, n-k_n)$ is regular if and only if there exists a $p \in [0,1]$ such that $\lim_{n \to \infty} {k_n}/{n} = p$. Furthermore, since the BD is regular of degree $2$, the unique MERW is simply the GRW, and it is associated with $m^\ast = \delta_{1/2}$. 

Indeed, let  $\mu_{1/2}$ be the distribution of the central Markov chain (starting from the root) corresponding to $\varphi_{1/2}\equiv \varphi_{\delta_{1/2}}$. Since all the paths of length $n$ starting from the root have the same probability $2^{-n}$,  the restriction of $\mu_{1/2}$ to $\mathcal T_n$ is equal to $\mu_n$. Necessarily  $\mu^\ast=\mu_{1/2}$.

\begin{rem}
	Note that the combinatorial identity (\ref{combipower0}) in this case is nothing but the well-known binomial expansion.
\end{rem}

\begin{rem}\label{rem:polya}
	It is interesting to observe that the central Markov chain corresponding to the uniform distribution $m$ on $[0,1]$ leads to the Pólya urn model with two colors. 
	
	This model is known to be an exchangeable process, exemplifying property (\ref{Unif}): all paths starting and ending at given points have the same probability when $w \equiv 1$.  	
\end{rem}

In light of Section \ref{section:connexion}, we can interpret $\mathbb{X}$ as the BD corresponding to the graph $G = \mathbb{Z}$. The unique $(1/r)$-harmonic function on $G$, with $0 < r \leq 1/2$ and $\psi(0) = 1$, is expressed as $\psi(x) = e^{\alpha x}$, where $2 \cosh(\alpha) = 1/r$. Applying the change of variable $x = 2k - n \in \mathbb{Z}$, we obtain
\begin{equation}
p^k (1-p)^{n-k} = \left(p(1-p)\right)^{\frac{n}{2}} \left(\frac{p}{1-p}\right)^{\frac{x}{2}}.
\end{equation}

Consequently, every extremal NNHF of the BD can be represented as in (\ref{product}) (even if the assumptions outlined in \cite{KestenRatio} are not fully met). 

To extend our analysis, introduce a weight function \(w\) into the aforementioned BD. Let us set \(w(x, y) = 1\) for all edges except those originating from the diagonal, where 
$$w((i, i), (i+1, i)) = \alpha \quad \text{and} \quad w((i, i), (i, i+1)) = \beta.$$

\begin{figure}[H]
	\centering
	\includegraphics[width=0.75\textwidth]{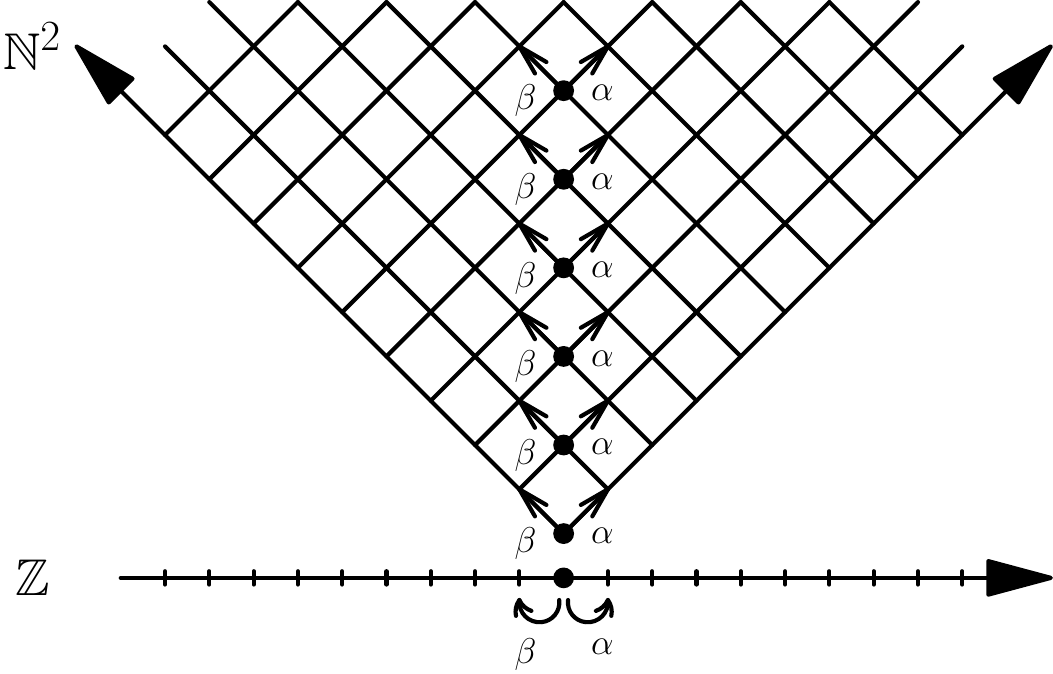}
	\caption{Weighted Pascal lattice}
	\label{RW}
\end{figure}

We refer to Figure \ref{RW}. In that case, describe the central measures and obtaining the corresponding MERWs presents a more complex challenge. Let us set
\begin{equation}\label{param}
\gamma=\frac{\alpha+\beta}{2},\quad s=\frac{1}{2}\sqrt{1-\left(1-\frac{1}{\gamma}\right)^2},\quad  \rho=s^{-1}\quad and \quad q=(2\gamma s)^{-1}. 
\end{equation}

\begin{prop}\label{prop:pascalweight}
There is a unique MERW on the latter weighted Pascal lattice. It is associated with the NNHF defined by
\begin{equation}\label{pascalharmo}
\varphi^\ast(n,x)=\frac{1+(\gamma^{-1}-1)|x|}{2^n}\quad (\text{if $\gamma <1$})\quad\mbox{or}\quad \varphi^\ast(n,x)=\frac{q^{|x|}}{\rho^n}\quad (\text{if $\gamma >1$}).
\end{equation}
Recall that for all $(k,n-k)\in \mathbb X_n$, we set $x=2k-n$. 
\end{prop}

\begin{proof} Let $A$ be the weighted adjacency matrix on $G=\mathbb Z$ defined as \( A(x,x\pm 1)=1 \) for \( x\neq 0 \), \( A(0,1)=\alpha \), \( A(0,-1)=\beta \), and \( A(x,y)=0 \) otherwise. 
	
Following (\ref{limitpoint}), to compute the MERWs, we need to study  the asymptotics of the weighted number of walks of length \( n \). To this end, introduce the generating functions
\begin{equation}
F_x(z)=\sum_{n\geq 0}\sum_{y\in \mathbb Z}A^n(x,y)z^n.
\end{equation}
One can check that \( F_{x}(z)=1+zF_{x+1}(z)+ zF_{x-1}(z) \) on \( \mathbb Z^\ast \) and \( F_0(z)=1+\alpha zF_1(z)+ \beta zF_{-1}(z) \). 

Thereafter, we can obtain
\begin{equation}
F_x(z)=\frac{1}{1-2z} -\frac{2(1-\gamma)z}{1-2z}\frac{(R(z))^{|x|}}{1-2\gamma z R(z)},\quad\text{where}\quad R(z)=\frac{1-\sqrt{1-4z^2}}{2z}.
\end{equation}
Assuming \(\gamma <1\), one can see that \( 1/2 \) is the smallest singularity of \( F_x(z) \), leading to 
\begin{equation}
F_x(z)\underset{z\to 1/2}{\sim} \sqrt{2}\left(\frac{\gamma}{1-\gamma}+|x|\right)\frac{1}{\sqrt{1-2z}}.
\end{equation}
On the contrary, when  \(\gamma>1 \), the smallest singularity occurs at $s\in(0,1/2)$ given in (\ref{param}) and one can check that  \( 2\gamma s R(s)=1 \) and
\begin{equation}
F_x(z)\underset{z\to s}{\sim} \frac{1}{2}\left(1-\frac{1}{\gamma}\right)^2\frac{R(s)^{|x|}}{1-2s}\frac{1}{s-z}.
\end{equation}

Since $1/(1-2z)$ is analytic on $\mathbb{C} \setminus \{1/2\}$ and $R(z)$ has an analytic continuation on $\mathbb{C} \setminus ]1/2,\infty[$, it follows that $F_x(z)$ admits an analytic continuation in some $\Delta$-domain (an open pacman-shaped region with the mouth corner located at the singularity $s$).  

Then, one can apply the transfer theorem, as detailed in \cite{Flajo,FlajoB}, and derive the asymptotics of (\ref{limitpoint}) in both cases, which imply \eqref{pascalharmo}.
\end{proof}

\noindent
\textit{iii) The Young Lattice.} Let us introduce the celebrated Young Lattice $\mathbb{Y} = \bigsqcup_{n \in \mathbb{N}} \mathbb{Y}_n$, a fundamental structure in combinatorics. This is the BD whose $n$th level set can be defined as
\begin{equation}
\mathbb{Y}_n = \left\{(n_1, n_2, \cdots) \in \bigcup_{k \geq 0} \mathbb{N}^k : n_1 \geq n_2 \geq \cdots \quad \text{and} \quad n_1 + n_2 + \cdots = n \right\}.
\end{equation}
Here $\mathbb N=\{1,2,\cdots\}$ and $\mathbb Y_0=\{\emptyset\}$.

As illustrated in Figure \ref{youngtabl}, each vertex of $\mathbb{Y}$ can be represented by Young diagrams and finite paths starting from the root in $\mathbb{Y}$ are represented by standard Young tableaux. For a comprehensive understanding, we can refer to  \cite{kerov, KerVer}.

\begin{figure}[H]
	\centering
	\includegraphics[width=0.5\textwidth]{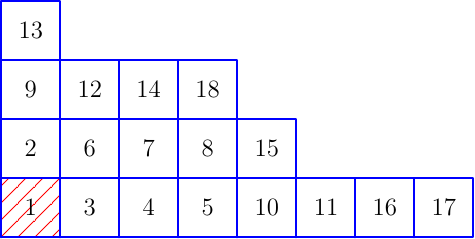}
	\captionsetup{width=0.8\textwidth}
	\caption{\it A standard Young tableau (in French notation) representing a path in $\mathbb{Y}$ from the root to the partition $8+5+4+1=18$.}
	\label{youngtabl}
\end{figure}

It is known that the boundary of $\mathbb{Y}$ is the Thoma simplex
\begin{equation}
\partial \mathbb{Y} = \left\{\alpha_1 \geq \alpha_2 \geq \cdots \geq 0; \beta_1 \geq \beta_2 \geq \cdots \geq 0 \;\bigg|\; \sum_{i \geq 1} \alpha_i + \sum_{i \geq 1} \beta_i \leq 1 \right\}.
\end{equation}

To explain this boundary, let us consider an infinite path $(\lambda_n)_{n \geq 0}$ in the Young lattice. The modified Frobenius coordinates of $\lambda_n$, denoted by $f_i^{(n)}$ and $g_i^{(n)}$, are depicted  in Figure \ref{frobe}. It can be shown that the path $\lambda_n$ is regular if and only if there exist $\alpha, \beta \in \partial \mathbb{Y}$ such that for all $i \geq 1$,
\begin{equation}
\lim_{n \to \infty} \frac{f_i^{(n)}}{n} = \alpha_i \quad \text{and} \quad \lim_{n \to \infty} \frac{g_i^{(n)}}{n} = \beta_i.
\end{equation}

\begin{figure}[H]
	\centering
	\includegraphics[width=0.55\textwidth]{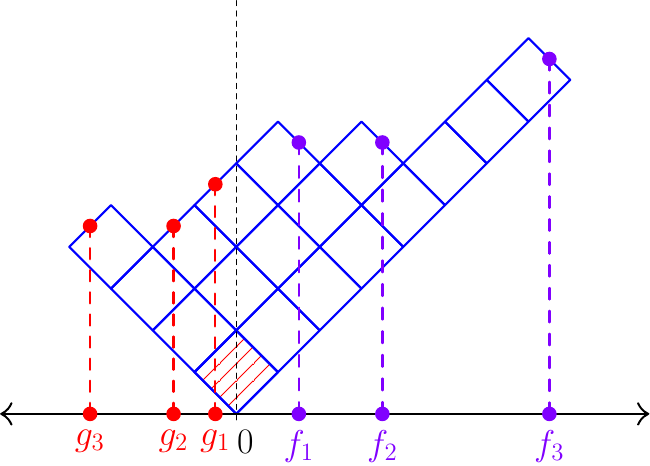}
	\captionsetup{width=0.8\textwidth}
	\caption{Modified Frobenius coordinates of a Young diagram (in Russian style). For example, $g_1 = 1/2$ and $f_1 = 3/2$.}
	\label{frobe}
\end{figure}

The Plancherel growth process corresponds to the central Markov chain associated with $\alpha = \beta \equiv 0$. It is the central measure given for any $\lambda$ of size $n$ (denote $\lambda \vdash n$) as
\begin{equation}\label{hook}
\mathbb{P}_n(\lambda) = \frac{d(\varnothing, \lambda)^2}{n!}, \quad \text{where} \quad d(\varnothing, \lambda) = \frac{n!}{\prod_{(i, j) \in \lambda} h_\lambda(i, j)}.
\end{equation}

The right-hand side of the latter equation is the celebrated Hook length formula, where $h_\lambda(i, j)$ denotes the hook length of the cell $(i, j)$ in $\lambda$. See for instance \cite{Romik_2015} for an overview.

\begin{prop}\label{plancherelgrowth}
	The unique MERW associated with the Young lattice is nothing but the Plancherel growth process.
\end{prop}

\begin{proof}
To begin with, the Gelfand measure, which corresponds to selecting a path of length $n$ uniformly and examining the Young diagram $\lambda$ obtained, is given by  
\begin{equation}
\mathbb{G}_n(\lambda) = \frac{d(\varnothing, \lambda)}{\sum_{\mu \vdash n } d(\varnothing, \mu)}.
\end{equation}
Hence, it coincides with the distribution $m_n$ introduced in (\ref{marginalunif}).  

Furthermore, it turns out that the (almost sure) limit shape of a large random Young diagram is identical under both the Plancherel and Gelfand measures. For further details, see \cite{Meliot}.  

Therefore, the proof follows from Corollary \ref{coro}.

\end{proof}


\section{Maximal entropy growing trees}

\label{section:tree}

\setcounter{equation}{0}

Following \cite{Bertoin}, consider the infinite genealogical tree 
\begin{equation}
\mathcal{I} = \bigcup_{n=0}^\infty \mathbb{N}^n,\quad\text{with $\mathbb{N}^0 = \{\emptyset\}$}.
\end{equation}

 The elements of $\mathcal{I}$ are sometimes called individuals. For any $i \in \mathcal{I}$, we denote by $|i|$ the integer $k \geq 0$ such that $i \in \mathbb{N}^k$. We say that $i$ belongs to the $k$-th generation. For any $i = (i_1, \cdots, i_k) \in \mathbb{N}^k$ and $j \in \mathbb{N}$, we set $ij = (i_1, \cdots, i_k, j)$. When $i=\emptyset$ (that is $k=0$) we simply write $ij=j$. The individual $ij$ is named the $j$-th child of $i$ and $i$ is the parent of $ij$.   
 
 More generally, we set $ij = (i_1, \cdots, i_k, j_1, \cdots, j_l)$ in an obvious meaning. The individual $ij$ is a descendant of $i$, and $i$ is an ancestor of $ij$.

 For a subset $\tau \subset \mathcal{I}$, we denote by $|\tau| \in \mathbb{N} \cup \{\infty\}$ its cardinal (named also its size). The following definition is borrowed from \cite{chauvin}.

\begin{defi}\label{def:prefix}
 A subset $\tau \subset \mathcal{I}$ is called a prefix tree if $\emptyset \in \tau$ and for all $i, j \in \mathcal{I}$, $i \in \tau$ whenever $ij \in \tau$. We denote by $\mathcal{P}_n$ the set of all prefix trees of size $n\geq 1$ and we set 
 \begin{equation}
 \mathcal{P} = \bigsqcup_{n \geq 0} \mathcal{P}_{n+1}.
 \end{equation}
 \end{defi}

\subsection{The BD structure of $\mathcal P$}

The set $\mathcal{P}$ of finite prefix trees is naturally endowed with a BD structure where the $n$th level set is  $\mathcal P_{n+1}$.  More precisely, one has $\tau \nearrow \sigma$ if and only if $\sigma = \tau \sqcup \{ij\}$ for some $i \in \tau$ and $j \in \mathbb{N}$ such that $ij \notin \tau$. Note that $\{\tau \to \sigma\} \neq \emptyset$ if and only if $\tau \subset \sigma$.

\begin{rem}
In the setting of Section \ref{sec:WBD}, the $n$th level set  of $\mathcal{P}$ consists of prefix trees of size $n+1$. For instance, $\varnothing = \{\emptyset\}$ is the root of the Bratteli diagram $\mathcal P$, its size is equal to $1$.
\end{rem}

For any $i \in \mathcal{I}$ and any prefix tree $\tau$, one can define the prefix tree of all the descendants of $i$ in $\tau$ by setting
\begin{equation}
\tau^{(i)} = \{j \in \mathcal{I} : ij \in \tau\}.
\end{equation}

In particular, the number of descendants of $i$ in $\tau$, including $i$, is given by $|\tau^{(i)}|$. In the case when $i \notin \tau$, one has $\tau^{(i)}=\emptyset$ and $|\tau^{(i)}|= 0$.

\begin{rem}
	Similarly to the correspondence between Young diagrams and standard Young tableaux illustrated in Figure \ref{youngtabl}, there is a one-to-one correspondence between paths in $\mathcal{P}$ and rooted trees with increasing labelings.
\end{rem}

To go further, the number of ordered increasing labelings of a given rooted tree $\sigma$ is 
\begin{equation}
d(\varnothing, \sigma) = \frac{|\sigma |!}{\displaystyle \prod_{i \in \sigma} |\sigma^{(i)}|}.
\end{equation}
Even though the proof follows easily by induction, we can refer to \cite[Chap. 5.1.4, Exer. 20]{Knuth1973}.

More generally, for any $\tau, \sigma \in \mathcal{P}$ with $\tau \subset \sigma$, it is straightforward to show that the combinatorial dimension between $\tau$ and $\sigma$ is given by  
\begin{equation}\label{hooktree}
d(\tau, \sigma) = \frac{|\sigma \setminus \tau|!}{\displaystyle \prod_{i \in \sigma \setminus \tau} |\sigma^{(i)}|!}.
\end{equation}

This equality is known as the hook-length formula for increasing labeled rooted trees, similar to that given in (\ref{hook}) for standard Young tableaux.

In particular, we obtain that, for all $\tau\subset \sigma \in \mathcal{P}$ of size $k$ and $n$ respectively, one has 
\begin{equation}\label{hook2}
\frac{d(\tau, \sigma)}{d(\varnothing, \sigma)} = \frac{\prod_{i \in \tau} |\sigma^{(i)}|}{n \cdots (n - k + 1)}.
\end{equation}
Then, by using the results in Appendix \ref{appendix} (see also Section \ref{sec:central}), one can easily describe the structure of the combinatorial Martin boundary of $\mathcal P$.

\begin{thm}\label{bratstructure}
	An infinite path $(\tau_n)_{n \geq 0}$ in $\mathcal{P}$ is regular if and only if, for all $i \in \mathcal{I}$, there exists $0 \leq \alpha_i \leq 1$ such that  
	\begin{equation}\label{asymp}
	\lim_{n \to \infty} \frac{|\tau_n^{(i)}|}{n} = \alpha_i.
	\end{equation}
	
	Consequently, there is a one-to-one correspondence between the Martin boundary $\partial \mathcal{P}$ and labeled infinite genealogical trees $(\alpha_i)_{i \in \mathcal{I}} \in [0,1]^{\mathcal{I}}$ satisfying
	\begin{equation}\label{eq:labelgenealogical}
	\alpha_\emptyset = 1\quad\text{and}\quad\forall i\in\mathcal{I},\; \alpha_i = \sum_{n \in \mathbb{N}} \alpha_{in}.
	\end{equation}  
	
	The corresponding extremal NNHFs and ergodic saturated central Markov kernels are then given respectively by
	\begin{equation}\label{NNHFtree}
	\varphi(\tau) = \prod_{i \in \tau} \alpha_i \quad \text{and} \quad p(\tau, \sigma) = \alpha_{ij},
	\end{equation}
	for any $\tau, \sigma \in \mathcal{P}$ with $\sigma = \tau \sqcup \{ij\}$, where $i \in \tau$ and $j \in \mathbb{N}$ with $ij \notin \tau$.
	
	In particular, we deduce from (\ref{combipower}) (equation \eqref{combipower0} in the introduction) the following combinatorial identity: for all $n\geq 0$ and arbitrary labels $(\alpha_i)_{i\in\mathcal{I}}$ satisfying (\ref{eq:labelgenealogical}),
	\begin{equation}\label{comb2}
	\sum_{\sigma \in \mathcal{P}_n} \frac{n!}{\displaystyle \prod_{i \in \sigma} |\sigma^{(i)}|!} \prod_{i \in \sigma} \alpha_i = 1.
	\end{equation}
\end{thm}

This allows us to reformulate this theorem as a multivariate analog of a hook length formula for trees, which can be viewed as another kind of multinomial expansion.
\begin{coro}
Let $\partial \tau$ denote the set of leaves of a prefix tree $\tau \in \mathcal P$.
For any $n\geq 0$ and for any vector of complex numbers $(\alpha_i)_{i \in \partial \tau}$, we have the following identity:
\begin{equation}
\sum_{\substack{\sigma \in \mathcal P_n\\ \sigma \subset \tau}} \left(
	\prod_{i \in \sigma} \dfrac{1}{|\sigma^{(i)}|}
	\sum_{j \in \tau^{(i)} \cap \partial \tau} \alpha_j
\right) = \dfrac{1}{n!} \left(
	\sum_{i \in \partial \tau} \alpha_i
\right)^n.
\end{equation}
\end{coro}
\begin{proof}
The statement follows from \eqref{comb2} by noticing that the $\alpha_i$ corresponding to each node can be expressed via the $\alpha_i$ of leaf vertices, and by dropping  $\alpha_{\emptyset} = 1$ to ensure homogeneity. 

More precisely, assume that $(\alpha_{i})_{i\in\partial \tau}$ is a non-negative family and let $(\alpha_{i})_{i\in\mathcal I}$ the unique extension  characterized by the  boundary values $(\alpha_{i})_{i\in\partial \tau}$, $\alpha_i=0$ for all $i\notin \tau$, and the relations: 
\begin{equation}
\forall i\in\overset{\circ}{\tau} := \{i \in \tau \mid \exists j \in \mathbb{N}, ij \in \tau\},\quad \alpha_i=\sum_{k\in\mathbb N}\alpha_{ik}.
\end{equation}
Note that  $\alpha_i=\sum_{j\in \tau^{(i)}\cap \partial \tau}\alpha_j$ for all $i\in\tau$. In particular, $\alpha_\emptyset=\sum_{j\in \partial \tau}\alpha_j$ and one can assume without lose of generality that $\alpha_\emptyset=1$ by homogeneity, in such way that \eqref{comb2} applies. 

Finally, because $\alpha_i=0$ as soon as $i\notin \tau$, one can restrict the sum over $\sigma\subset \tau$. Since this identity is true for positive values of $\alpha_i$, and both sides are multivariate polynomials, it holds for all complexes values as well.
\end{proof}

\subsection{Underlying skeleton associated with a boundary point}

For any $\alpha \in \partial \mathcal{P}$, we denote by $\mu_\alpha$ the corresponding saturated ergodic central measure and by $\mathbb{S}_\alpha$ its the support (again, we refer to Appendix \ref{appendix} for formal definitions). 

\begin{defi}
We call the support of $\alpha$, that is $\mathcal{S}_\alpha=\{i\in \mathcal I : \alpha_i>0\}$, the skeleton of $\mu_\alpha$ or, equivalently, the skeleton of the associated saturated ergodic central Markov chain $\mu_\alpha$. 	
\end{defi}
 
Roughly speaking, this infinite prefix tree serves as the skeleton upon which the trees given by the saturated ergodic central Markov chain, denoted by $(\tau_n)_{n \geq 0}$, evolves.

In particular, by using \eqref{asymp}, one has 
\begin{equation}
\mu_\alpha\left(\bigcup_{n=0}^\infty \uparrow \tau_n = \mathcal{S}_\alpha\right) = 1\quad\text{and}\quad \mathbb S_\alpha=\{\tau\in\mathcal P : \tau\subset \mathcal S_\alpha\}.
\end{equation}

One can ckeck that $\mathcal S_\alpha$ is an infinite prefix tree such that  	
\begin{equation}\label{skeleton}
\forall i\in\mathcal S_\alpha,\;\exists k\in\mathbb N,\; ik\in\mathcal S_\alpha. 
\end{equation}

Reciprocally,  given an infinite prefix tree $\mathcal S$ satisfying \eqref{skeleton}, one can easily check that there exists at least one $\alpha\in\mathcal P$ such that  $\mathcal S_\alpha=\mathcal S$. This motivates the following definition.

\begin{defi}
An infinite prefix tree satisfying \eqref{skeleton} is called a skeleton. 
\end{defi} 
 
\subsection{Correspondence with fragmentation processes}

There exists a correspondence between the boundary $\partial \mathcal{P}$ and fragmentation processes. 

To clarify this relationship, let $\mathcal S$ be an arbitrary   skeleton and consider $(\nu_i)_{i \in \mathcal{\mathcal S}}$ a family of probability measures where, for each $i \in \mathcal{S}$, 
\begin{equation}
{\rm supp}(\nu_i)=\{j\in \mathcal S : \exists k\in\mathbb N,\, j=ik\}.
\end{equation}
Then, one can associate an element $\alpha \in \partial \mathcal{P}$, having $\mathcal S$ for skeleton, by setting $\alpha_\emptyset = 1$ and, 
\begin{equation}\label{frag}
\forall i \in \mathcal{S},\; \forall j \in {\rm supp}(\nu_i),\; \alpha_j := \nu_i(j) \alpha_i.
\end{equation}

Reciprocally, given  $\alpha\in\mathcal P$ arbitrary, the construct the corresponding family of framentation measures $(\nu_{i})_{i\in\mathcal S_\alpha}$, it suffices to set for all $i\in\mathcal S_\alpha$ and $j\in\mathcal S_\alpha$ a child of $i$,  $\nu_i(j)=\alpha_{j}/\alpha_i$

Also, introducing $C_i\subset\mathbb N$ as the set of $k\in\mathbb N$ such that $ik\in\mathcal S$, each probability measure $\nu_i$ can be seen as a probability measure on $\mathbb N$ whose support is $C_i$ by setting $\nu_i(k):=\nu_i(ik)$. 

We refer to Figure \ref{Frag} for an illustration. 

\begin{figure}[H]
	\centering
	\includegraphics[width=0.65\textwidth]{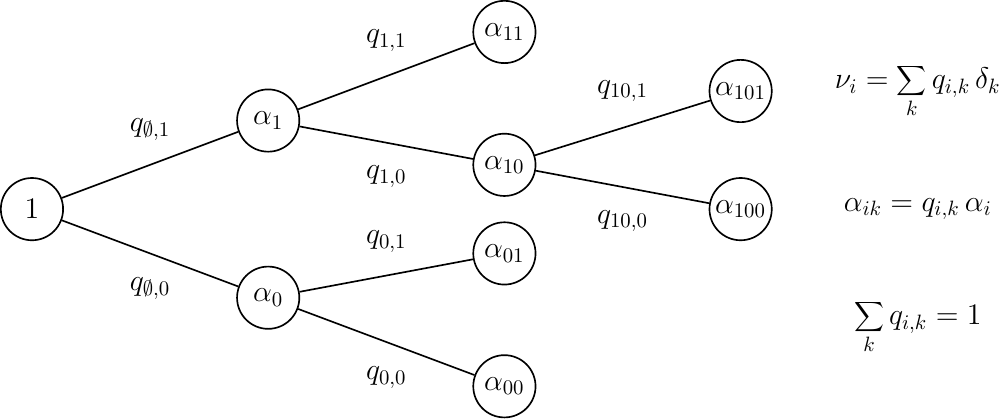}
	\captionsetup{width=0.8\textwidth}
	\caption{Correspondence between the labels  $(\alpha_i)_{i\in \mathcal I}$ and the fragmentation probability measures $(\nu_i)_{i\in\mathcal I}$.}
	\label{Frag}
\end{figure}

\subsection{About MERWs}

Since $d(\varnothing, \mathcal{P}_n) = \infty$ for all $n \geq 1$, Assumption \ref{Ass2} is not satisfied, thereby precluding the existence of a MERW on the BD of prefix trees. To address this issue, we propose two alternatives:
\begin{enumerate}
	\item Introduce a weight structure on $\mathcal{P}$ such that Assumption \ref{Ass2} is fulfilled.
	\item Select a skeleton $\mathcal{S}$ where $d(\varnothing, \mathcal{P}_n^{\mathcal{S}}) < \infty$ for all $n \geq 1$, with  
	\begin{equation}
	\mathcal{P}_n^{\mathcal{S}} = \{\tau \in \mathcal{P}_n : \tau \subset \mathcal{S}\}.
	\end{equation}
	Then, investigate the MERWs on $\mathcal{P}^{\mathcal{S}} = \bigsqcup_{n \geq 0} \mathcal{P}_{n+1}^{\mathcal{S}}$. In this scenario, we also refer to MERWs evolving in the skeleton $\mathcal{S}$, which can be seen as an additional constraint.
\end{enumerate}

\section{The example of $d$-ary trees}

\setcounter{equation}{0}

In this section, we  present several examples involving $d$-ary trees, that is, the case when the underlying skeleton is given by the complete infinite $d$-ary tree 
\begin{equation}
\mathcal{S}_d = \{i=(i_1i_2\cdots) \in \mathcal{I} : \forall 1 \leq k \leq |i|, \; 1 \leq i_k \leq d\}.
\end{equation}
Equivalently, we assume in this section that the trees grow into $\mathcal S_d$.

%
%
%
%

\label{sec:examples}

\subsection{Uniform fragmentation measure} \label{sec:uniffrag}

We first consider the basic model obtained when the fragmentation probability measures (\ref{frag}) are uniform over $\{1, \cdots, d\}$.

One can easily check that the corresponding boundary point $\alpha \in \partial \mathcal{S}_d$ is given for any $i \in \mathcal{S}_d$ by $\alpha_i = d^{-|i|}$. Note that for any finite $d$-ary tree $\tau$, 
\begin{equation}\label{newphi}
\prod_{i\in\tau}\alpha_i=\prod_{i\in\tau} \frac{1}{d^{|\tau^{(i)}|-1}}.
\end{equation}

This model of tree growth is discussed in \cite{Pittel1} and also appears in \cite{IDLAtree, Aldous} as a specific instance of an Internal Diffusion-Limited Aggregation (IDLA) process on a tree.

Surprisingly, the general combinatorial identity (\ref{comb2}) (with the help of \eqref{newphi}) enables us to derive, as a special case, Han's hook length formula for binary trees, as found in \cite{Han1}.

\begin{coro}\label{Han-d}
	Let $\mathcal{B}_d(n)$ denote the set of all $d$-ary trees with $n$ vertices. Then 
	\begin{equation}\label{comb3}
	\sum_{\tau \in \mathcal{B}_d(n)} \frac{n!}{\displaystyle \prod_{v \in \tau} |\tau^{(v)}| d^{|\tau^{(v)}|-1}} = 1.
	\end{equation}
\end{coro}

\subsection{Dirichlet random environment.} 

\label{sec:randomenv}

Let us consider $(U_{i,1}, \cdots, U_{i,d})_{i \in \mathcal{S}_d}$ a family of \emph{i.i.d.\@} Dirichlet random variables with positive parameters $a_1, \cdots, a_d$: \emph{i.e.\@} having for density on the $(d-1)$-simplex 
\begin{equation}
\{(x_1,\cdots,x_d): x_1+\cdots+x_d=1,\; x_1,\cdots,x_d\geq 0\},
\end{equation}
the function
\begin{equation}\label{density}
\frac{1}{\mathbf B(a_1,\cdots,a_d)}
\prod_{i=1}^d x_i^{a_i-1}, 
\end{equation}
where $\mathbf B$ denotes the usual multivariate beta function. 

Then define the random fragmentation probability measures by
\begin{equation}
\nu_i(\omega) = \sum_{j=1}^d U_{i,j}(\omega) \delta_{ij}, \quad i \in \mathcal{S}_d.
\end{equation}

It corresponds to these random fragmentation measures  a random point $\alpha(\omega)$ in the Martin boundary $\partial \mathcal{S}_d$ (see Figure \ref{Frag}). Specifically, for all $n \geq 1$ and $i_1, \cdots, i_n \in \{1, \cdots, d\}$, one can check that 
\begin{equation}
\alpha_{i_1 \cdots i_n} = U_{\emptyset, i_1} U_{i_1, i_2} U_{i_1 i_2, i_3} \cdots U_{i_1 \cdots i_{n-1}, i_n}.
\end{equation}

It follows (see \eqref{NNHFtree} and Figure \ref{Fragenv} for an example) that the corresponding random NNHF on $\mathcal{S}_d$ is given by 
\begin{equation}\label{diri}
\varphi_\omega(\tau) = \prod_{i \in \overset{\circ}{\tau}} \prod_{j=1}^d U_{i,j}(\omega)^{|\tau^{(ij)}|},
\end{equation}
where $\overset{\circ}{\tau} = \{i \in \tau \mid \exists j \in \mathbb{N}, ij \in \tau\}$ denotes the set of internal nodes of a prefix tree. 

\begin{figure}[H]
	\centering
	\includegraphics[width=0.6\textwidth]{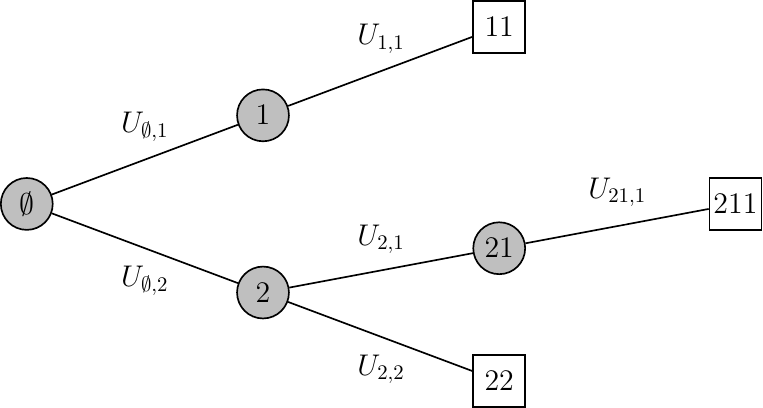}
	\captionsetup{width=0.8\textwidth}
	\caption{A binary tree $\tau$ where the leaves $\partial \tau$ are represented by squares and the internal nodes $\overset{\circ}{\tau}$ by filled circles. Here, $\varphi(\tau)=U_{\emptyset,1}^2 U_{1,1}U_{\emptyset,2}^4 U_{2,1}^2U_{21,1} U_{2,2}$.}
	\label{Fragenv}
\end{figure}

\noindent
\textit{i) Annealed environment and corresponding IDLA-like random walk.}  In light of the Pólya urn process, which can be obtained by the taking the mean over the uniform probability measure of the extremal harmonic functions on the Pascal lattice  (see Remark \ref{rem:polya}), one can consider the PHF on $\mathcal{S}_d$ defined by $\overline{\varphi} = \mathbb{E}[\varphi]$.  

Then, simple calculations leads to
\begin{equation}\label{mean}
\overline{\varphi}(\tau) = \prod_{i \in \overset{\circ}{\tau}} \frac{\mathbf B(a_1 + |\tau^{(i1)}|, \cdots, a_d + |\tau^{(id)}|)}{\mathbf B(a_1, \cdots, a_d)},
\end{equation}
and we obtain that the central Markov kernel $\overline p$ associated with $\overline \varphi$ is given by 
\begin{multline}\label{transidla}
\overline p(\tau, \tau \sqcup \{j_1 \cdots j_{k-1}j_k\}) \\
=\frac{a_{j_1} + |\tau^{(j_1)}|}{A + |\tau| - 1} \frac{a_{j_2} + |\tau^{(j_1 j_2)}|}{A + |\tau^{(j_1)}| - 1} \cdots 
\frac{a_{j_2} + |\tau^{(j_1\cdots j_{k-1})}|}{A + |\tau^{(j_1\cdots j_{k-2})}| - 1}
\frac{a_{j_k}+0}{A + |\tau^{(j_1 \cdots j_{k-1})}| - 1},
\end{multline}
for every $k \geq 1$ and $j_1, \cdots, j_k \in \{1, \cdots, d\}$ such that $j_1 \cdots j_{k-1} \in \tau$ and $j_1 \cdots j_{k-1} j_k \notin \tau$, and with $A=a_1+\cdots+a_d$. These transition probabilities  have an IDLA-like interpretation.

Let $(\tau_n)_{n\geq 0}$ be the corresponding central Markov chain on $\mathcal S_d$. To obtain $\tau_{n+1}$ from $\tau_n$, it suffices to perform an appropriate  RW on $\mathcal{S}_d$, starting from the root, until exiting $\tau_n$, and then add the newly visited node. More precisely, at each node $i$ of $\tau_n$, we have to choose a child $ij$, $j\in\{1,\cdots,d\}$, with a probability proportional to $a_j + |\tau^{(ij)}_n|$. Recall that $\tau^{(ij)}_n=\emptyset$ when $ij\notin \mathcal \tau_n$, that is $|\tau^{(ij)}_n|=0$. We refer to Figure \ref{Fragidela}.\\ 

\begin{figure}[H]
	\centering
	\includegraphics[width=0.55\textwidth]{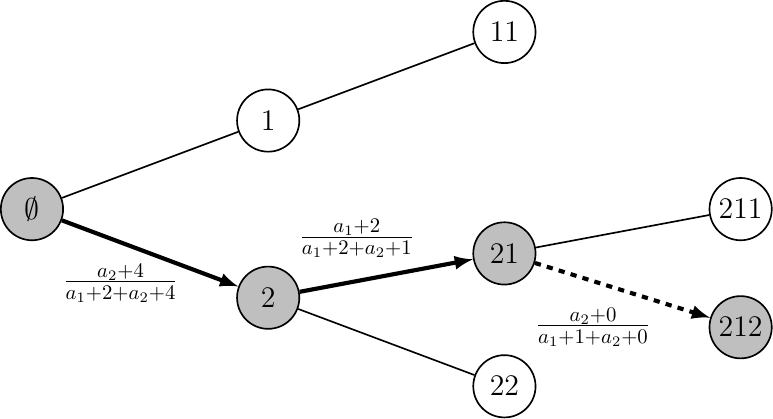}
	\captionsetup{width=0.8\textwidth}
	\caption{IDLA-like random walk. Example with $d=2$ (binary trees). The probability to add the node $212$ is given by the product of the transition probabilities appearing on the directed edges.}
	\label{Fragidela}
\end{figure}

\noindent
\textit{ii) The $d$-ary search tree process.} In the case when the $a_j$ above are all equal to $1/(d-1)$ and $|\tau| = n$, we can rewrite (\ref{transidla}) as 
\begin{equation}\label{MERWdary}
p(\tau, \tau \sqcup \{v\}) = \frac{1}{1 + (d-1)n},
\end{equation}
for any of the $1 + (d-1)n$ available vertices $v$ such that $\tau \nearrow \tau \sqcup \{v\}$. 

The case when $d = 2$ corresponds to the usual Binary Search Tree (BST) process and so we refer to this process as the $d$-ary search tree process.

By using Theorem \ref{convcentral}, we can derive directly some well-known asymptotics which can be found in \cite[Chap. 6.2.2]{chauvin}. 
\begin{coro}\label{coro:BST}
Let $(\tau_n)_{n\geq 0}$ be the $d$-ary search tree process. For any $i \in \mathcal{S}_d$ and $j \in \{1, \ldots, d\}$,
\begin{equation}\label{BSTasymp}
\lim_{n \to \infty} \frac{|\tau^{(ij)}_n|}{|\tau^{(i)}_n|} \overset{a.s.}{=} U_{i,j},
\end{equation}	
where the $(U_{i,1}, \cdots, U_{i,d})_{i \in \mathcal{S}_d}$ a family of \emph{i.i.d.\@} Dirichlet random variables with parameters 
\begin{equation*}
a_1=\cdots=a_d=\frac{1}{d-1}.
\end{equation*}

\end{coro}

\begin{prop}\label{BST}
The unique MERW relative to the complete infinite $d$-ary tree $\mathcal{S}_d$ is the $d$-ary search tree process.
\end{prop}

\begin{proof}
Since at each step the transition probability (\ref{MERWdary}) depends only on the size of the current $d$-ary tree, all the paths of length $n$ starting from the root  have for probability 
\begin{equation*}
\prod_{k=1}^{n}\frac{1}{1+(d-1)k}.
\end{equation*}
Similarly to  {\it ii)} Section \ref{sec:illustra}, we obtain that the  $d$-ary search tree process is the unique  MERW.

\end{proof}

\subsection{Adding preferential weights} 

To go further, we can consider weighted $d$-ary trees where the weight assigned to adding the $j$th child of a node $i$ for the first time among the other children is denoted by $w_{i,ij}$. 

For a finite path $T\in\{\varnothing\to\tau\}$ starting from the root and ending at $\tau$, define for each $j \in \{1, \ldots, d\}$ the set $N_j$ as the collection of all internal nodes $i\in\overset{\circ}{\tau}$ for which $ij$ is the first child of $i$ appearing in the trajectory. In other words, among the children of $i$ appearing in $\tau$, $ij$ is the one with the smallest label. Observe that $N_1, \ldots, N_d$ form a partition of $\overset{\circ}{\tau}$.

The weight of $T$ can be expressed as  (see Figure \ref{Fragweights} for an example)
\begin{equation}\label{weight}
w_T = \prod_{j=1}^d \prod_{i \in N_j} w_{i,ij}.
\end{equation}

\begin{figure}[H]
	\centering
	\includegraphics[width=0.65\textwidth]{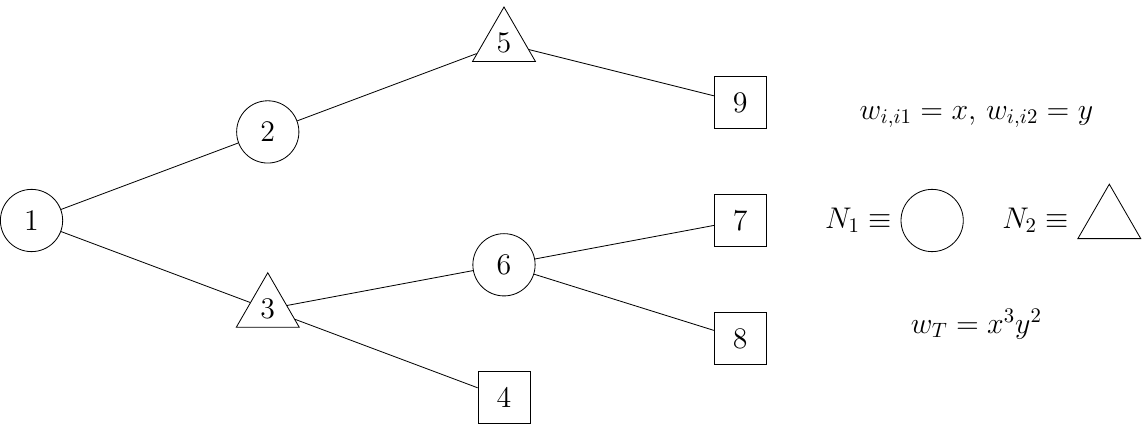}
	\captionsetup{width=0.8\textwidth}
	\caption{A path $T$ of length $8$ in the infinite rooted binary tree, where $x$ (resp. $y$) is the weight assigned when adding the first (resp. second) child of any node for the first time.
	}
	\label{Fragweights}
\end{figure}

\begin{lem}\label{lem:dimpref} For all finite $d$-ary tree $\sigma$, one has
\begin{equation}\label{countpref}
d(\varnothing, \sigma) = \frac{|\sigma|!}{\displaystyle \prod_{i \in \sigma} |\sigma^{(i)}|} \prod_{i \in \overset{\circ}{\sigma}} \frac{\sum_{j=1}^d |\sigma^{(ij)}| w_{i, ij}}{|\sigma^{(i)}| - 1}.
\end{equation}

More generally, for any $d$-arry tree $\tau$ included into $\sigma$,
\begin{equation}\label{countpref2}
d(\tau, \sigma) = \frac{|\sigma \setminus \tau|!}{\prod_{i \in \sigma \setminus \tau} |\sigma^{(i)}|} \prod_{i \in (\overset{\circ}{\sigma} \setminus \tau) \sqcup (\overset{\circ}{\sigma} \cap \partial \tau)} \frac{\sum_{j=1}^d |\sigma^{(ij)}| w_{i, ij}}{|\sigma^{(i)}| - 1}.
\end{equation}
\end{lem}

\begin{proof} 	
To compute the combinatorial dimension of such increasing labeled trees, we use the \emph{boxed product} method from \cite[II.6.3]{FlajoB}.

A \emph{boxed product} of two labeled combinatorial classes $\mathcal A$ and $\mathcal B$ (in our case, sets of prefix trees), denoted as $\mathcal C = \mathcal A^\square \star \mathcal B$, is a class consisting of pairs of objects $(a,b) \in \mathcal A \times \mathcal B$ such that the minimal label of $a$ is less than the minimal label of $b$. 

The number of such pairs $(a,b)$ such that $a$ and $b$ are of size $k$ and $n-k$ respectively can be counted by the formula
\begin{equation}\label{boxprod}
\binom{n-1}{k-1} A_k B_{n-k} =
\frac{k}{n}\binom{n}{k}  A_k B_{n-k},
\end{equation}
where $A_k$ and $B_{n-k}$ denote the number of objects of size $k\in\mathbb N$ and $n-k\in\mathbb N$ from $\mathcal A$ and $\mathcal B$ respectively. 

Therefore, assuming that the first child of the root is $j \in \{1,\dots,d\}$ (equivalently, that the minimal label of $\sigma^{(j)} \sqcup \bigsqcup_{l\neq j} \sigma^{(l)}$ belongs to the $j$th component), and using this counting principle, one get the recursive formula
\begin{equation}\label{countrec}
d(\varnothing, \sigma) = \left(\sum_{j=1}^d \frac{|\sigma^{(j)}|}{|\sigma| - 1} w_{\emptyset, j}\right)
\binom{|\sigma| - 1}{|\sigma^{(1)}|, \ldots, |\sigma^{(d)}|} \prod_{k=1}^d d_k(\varnothing, \sigma^{(k)}),
\end{equation}
where $d_k$ denotes the combinatorial dimension in which the weights $w_{i, ij}$ are replaced by $w_{ki, kij}$, and $\binom{n}{k_1,\cdots,k_d}$ denotes the usual multinomial coefficient. 

Thereafter, the proof of the lemma follows by induction.	
\end{proof}

By using Lemma \ref{lem:dimpref}, we obtain results similar to those in Theorem \ref{bratstructure}. The proof follows from simple calculations and is omitted.
 
\begin{coro}\label{Bratstructure2}
The normalized extremal NNHFs of this WBD take the form 
\begin{equation}\label{harmpref}
\varphi(\tau) = {\prod_{i \in \tau} \alpha_i} \, {\prod_{i \in \overset{\circ}{\tau}} \frac{\alpha_i}{\sum_{j=1}^d \alpha_{ij} w_{i, ij}}},
\end{equation}
with $\alpha_\emptyset = 1$ and $\alpha_i = \alpha_{i1} + \ldots + \alpha_{id} \in [0, 1]$ for any $i \in \mathcal{S}_d$. Again, the $\alpha_i$  represents the asymptotic proportion of descendants of $i$ in a regular path.

Besides, the corresponding saturated  ergodic central Markov kernel is then given by
\begin{equation}
p(\tau, \tau \sqcup \{ij\}) = \left\{
\begin{array}{cl}
\displaystyle \frac{\alpha_{ij} w_{i, ij}}{\sum_{k=1}^d \alpha_{ik} w_{i, ik}} \alpha_i  , & \text{if } i \in \partial \tau, \\[15pt]
\alpha_{ij}, & \text{if } i \notin \partial \tau \text{ and } ij \notin \tau,
\end{array}
\right.
\end{equation}
for any $i \in \tau$ and $j \in \{1, \ldots, d\}$.
 
\end{coro}

As a consequence, we obtain the following generalization of Han's hook length formula (\ref{comb3}). 	

\begin{coro} 
Let $\mathcal{B}_d(n)$ be the set of all $d$-ary trees with $n$ vertices. Then
\begin{equation}\label{comb4}
\sum_{\tau \in \mathcal{B}_d(n)} \frac{\prod_{v \in \overset{\circ}{\tau}} \frac{\sum_{j=1}^d |\tau^{(vj)}| w_j}{|\tau^{(v)}| - 1}}{\displaystyle \prod_{v \in \tau} |\tau^{(v)}| d^{|\tau^{(v)}| - 1}}  \left(\frac{w_1 + \cdots + w_d}{d}\right)^{|\partial \tau|} = \frac{1}{n!} \left(\frac{w_1 + \cdots + w_d}{d}\right)^n.
\end{equation}
\end{coro}

\begin{proof}
Consider the uniform fragmentation measure case  as in Section \ref{sec:uniffrag}, that is $\alpha_i = d^{-|i|}$ and $w_{i, ij} \equiv w_j$ for all $i\in\mathcal S_d$ and $1\leq j\leq d$.

By using  \eqref{newphi} and  by noting that
\begin{equation}
{\prod_{i \in \overset{\circ}{\tau}} \frac{\alpha_i}{\sum_{j=1}^d \alpha_{ij} w_{i, ij}}}=\left(\frac{d}{\sum_{j=1}^dw_j}\right)^{|\tau|- |\partial \tau|},
\end{equation}
we deduce the result from \eqref{harmpref} and the general combinatorial identity  \eqref{combipower}. 
\end{proof}

As an example,  when $d = 2$ and $n = 3$, we obtain
\begin{equation*}
\frac{1}{6} \left(\frac{w_1 + w_2}{2}\right)^3 = \frac{1}{12} \left(\frac{w_1 + w_2}{2}\right)^3 + \frac{1}{48} \sum_{1 \leq i, j \leq 2} w_i w_j \left(\frac{w_1 + w_2}{2}\right).
\end{equation*}

It is much more difficult to describe the MERWs in full generality. However, drawing on the previous situation, one can construct them using appropriate IDLA-like random walks under some additional assumptions.

\begin{prop}\label{prop:sum1}
	Assume that $\sum_{k=1}^d w_{i, ik} = d$ for all $i \in \mathcal{S}_d$. Then there is a unique MERW. 
	
	The Markov kernel can be written, for any $d$-ary tree $\tau$ of size $n$, $i \in \tau$ and $j \in \{1, \cdots, d\}$ such that $ij \notin \tau$, by 
	\begin{equation}
	p(\tau, \tau \sqcup \{ij\}) = \left\{
	\begin{array}{cl}
	\frac{1}{1+(d-1)n}, & \text{if } i \notin \partial \tau, \\[15pt]
	\frac{w_{i, ij}}{1+(d-1)n}, & \text{if } i \in \partial \tau.
	\end{array}
	\right.
	\end{equation}
\end{prop}

\begin{proof} We define an increasing sequence of $d$-ary trees $(\tau_n)_{n\geq 0}$ as follows. Set $\tau_0=\{\emptyset\}$ and, given $\tau_n$,  we perform a random walk on $\mathcal S_d$. 
	
The transitions of this random walk are such that at each node $i$ of $\tau_n$, we choose a child $ij$, where $j\in\{1,\cdots,d\}$, with a probability equal to 
\begin{equation}
\frac{1 + (d-1)|\tau^{(ij)}_n|}{\sum_{k=1}^d 1 + (d-1)|\tau^{(ik)}_n|}=\frac{1 + (d-1)|\tau^{(ij)}_n|}{d + (d-1)(|\tau^{(i)}_n|-1)}= \frac{1 + (d-1)|\tau^{(ij)}_n|}{1 + (d-1)|\tau^{(i)}_n|},
\end{equation}
when $i$ is not a leaf, and with a probability proportional to $w_{i, ij}$ otherwise. Once there exits $\tau_n$, we add the newly visited node $v$ and we set $\tau_{n+1} = \tau_n \sqcup \{v\}$.

Denote by $\mu$ the distribution of the corresponding random walk. Let $T$ be any path starting from the root and ending at $\tau$ of size $n$ and denote as previously by $N_j$ the set of internal nodes $i$ of $\tau$ such that $ij$ is the first child of $i$ appearing in the path. 

Observe that the probability of such a path $T$ is given by   
\begin{equation}
\mu(T)=\prod_{k=0}^{n-1} \frac{1}{1 + (d-1)k} \prod_{j=1}^d \prod_{i \in N_j} \frac{d \, w_{i, ij}}{\sum_{k=1}^d w_{i, ik}}.
\end{equation}  
Here, we use the fact that $1+(d-1)|\tau^{(k)}|=d$ when $k\in\partial \tau$. 

Noting that the latter probability is proportional to the weight $w_T$ defined in \eqref{weight}, in the case when $\sum_{k=1}^d w_{i,ik} \equiv d$, we deduce that the distribution $\mu_n$ in \eqref{mun} is, in this case, simply the restriction of $\mu$ to $\mathcal{T}_n$, leading to the desired result.
\end{proof}

In the sequel, we assume that $d=2$ but also that that there exists $x,y>0$ such that $w_{i,1}=x$ and $w_{i,2}=y$ for all $i\in\mathcal S_2$ as in Figure \ref{Fragweights}.

\begin{thm}\label{thm:BSTgen}
There exists a unique MERW. The corresponding PHF is given for all binary tree $\tau$ by
\begin{equation}\label{eq:harmmerwbstgen}
\varphi(\tau)=\frac{1}{|\tau|!} \left(\frac{2}{x+y}\right)^{|\overset{\circ}{\tau}|},
\end{equation}

In particular, the non-zero transition probabilities of the MERW are given by
\begin{equation}
p(\tau,\tau\sqcup{\{ij\}})=
\left\{\begin{array}{cl}
\frac{2x}{x+y}\frac{1}{|\tau|+1},&\text{if $i\in\partial \tau$ and $j=1$,}\\[10pt]
\frac{2y}{x+y}\frac{1}{|\tau|+1},& \text{if $i\in\partial \tau$ and $j=2$,}\\[10pt]
\frac{1}{|\tau|+1},& \text{if $i\in \overset{\circ}{\tau}$ and $ij\notin \tau$.}
\end{array}\right.
\end{equation}

Furthermore, let  $(\tau_n)_{n\geq 0}$ be the corresponding MERW starting from the root. Then, for all $i\in\mathcal S_2$ and $j\in \{1,2\}$,
\begin{equation}\label{eq:asympBSTgen}
\lim_{n\to\infty}\frac{|\tau_n^{(ij)}|}{|\tau_n^{(i)}|}\overset{\text{p.s.}}{=}V_{i,j},
\end{equation}
where $(V_{i,1},V_{i,2})_{i\in\mathcal S_2}$ are {i.i.d.\@} random variables on the  simplex $\{(t,1-t) : 0\leq t\leq 1\}$ having for density the function
\begin{equation}\label{eq:density}
h_{x,y}(t)=\frac{2xt+2y(1-t)}{x+y}.
\end{equation}
\end{thm}

\begin{rem}
	Note that the case when $s=x+y=2$ is covered by Proposition \ref{prop:sum1}.
\end{rem}

\begin{rem}
	It is well known that the height $H_n$ of a random BST of size $n$ is of the order $c \log(n)$ for some positive constant $c$, as $n$ goes to infinity (see Théorème 6.22 in \cite{chauvin}).  
	
	This case corresponds to $x = y = 1$, and it could be interesting to study the asymptotics of the height of the MERW for arbitrary $x, y > 0$.  
\end{rem}

\begin{rem}
	A similar result can be stated and proved for weighted $d$-ary trees, where the MERW turns out to be given by 
	\begin{equation}
	p(\tau,\tau\sqcup{\{ij\}})=
	\left\{\begin{array}{cl}
	\frac{d\, x_j}{\sum_{k=1}^d x_k}\frac{1}{1+(d-1)|\tau|}, & \text{if } i\in\partial \tau \text{ and } j\in\{1,\dots,d\}, \\[10pt]
	\frac{1}{1+(d-1)|\tau|}, & \text{if } i\in \overset{\circ}{\tau} \text{ and } ij\notin \tau.
	\end{array}\right.
	\end{equation}
	Here, $x_j$ denotes the preferential attachment of the $j$th child of a leaf for the first time. 
	
We observe that these transition probabilities are homogeneous with respect to these parameters, in contrast to the weights of the paths (see, for instance, the example in Figure \ref{Fragweights}).

\end{rem}

\begin{proof}
We first study the asymptotics of the total combinatorial dimension $T_{n+1} = d(\varnothing, \mathbb{X}_n)$. 

Recall that $\mathbb{X}_n$ represents the set of binary trees of size $n+1$, in such a way that $T_n$ corresponds to the total weight of increasing labeled trees of size $n$, for $n \geq 1$. 

For instance, we have $T_1 = 1$, $T_2 = x + y$ and $T_3 = x^2 + 2xy + y^2 + x + y$.

\begin{lem}\label{lem:rec}
For all $n\geq 1$, one has 
\begin{equation}\label{eq:rec}
T_{n+1}=(x+y)\left[T_n+\sum_{k=1}^{n-1}\binom{n-1}{k-1} T_k T_{n-k}\right].
\end{equation}
\end{lem} 

\begin{proof}[Proof of Lemma \ref{lem:rec}]
The lemma relies on the following observation: a binary tree of size $n+1$ can be either  
\begin{enumerate}  
	\item a tree whose root has only one child;  
	\item a tree whose root has two children.  
\end{enumerate}  

The first case corresponds to the term $T_n$, multiplied by $x$ or $y$ depending on whether the first or second offspring of the root is present.  

In the second case, we distinguish which offspring $c_1 \in \{1,2\}$ was created first, corresponding to multiplication by $x$ and $y$. In each of these two cases, let $k\in\{1,\cdots, n-1\}$ be the size of the subtree containing all descendants of $c_1$. Writing $\{1,2\}$ as  $\{c_1, c_2\}$, the subtree containing all descendants of $c_2$ must necessarily have size $n-k\in \{1,\cdots, n-1\}$.  We then need to distribute $n$ labels between these two increasing labeled rooted trees, with the constraint that the subtree of size $k$ contains the smallest one. The number of such choices is given by $\binom{n-1}{k-1}$, similarly to the Boxed product (\ref{boxprod}).  

This completes the proof.
\end{proof}

In the following, we set $s=x+y$ and we introduce the exponential generating function 
\begin{equation}
\mathfrak T(u)=\sum_{n=1}^{\infty} T_n\frac{u^n}{n!}.
\end{equation}

\begin{lem}\label{lem:edo}
The recurrence \eqref{eq:rec} translates into the language of generating functions as
\begin{equation}\label{eq:ODE}
\partial_u \mathfrak T(u)=1+s\mathfrak T(u)+s\int_0^u \mathfrak T(v)\partial_v\mathfrak T(v)dv.
\end{equation}
\end{lem}

\begin{proof}[Proof of Lemma \ref{lem:edo}]
Noting that $\mathfrak T(u)=u+\sum_{n\geq 2} T_n\frac{u^n}{n!}$ and using (\ref{eq:rec}), we obtain 
\begin{eqnarray}
\partial_u\mathfrak T(u) & = & 1+\sum_{n\geq 2}T_n\frac{u^{n-1}}{(n-1)!}\\
& = & 1+s\sum_{n\geq 2} T_{n-1}\frac{u^{n-1}}{(n-1)!}+s\sum_{n\geq 3}\left(\sum_{k=1}^{n-2}\binom{n-2}{k-1}T_kT_{n-1-k}\right)\frac{u^{n-1}}{(n-1)!}.\label{eq:partial}
\end{eqnarray}
Since the first sum in \eqref{eq:partial} is equal to $\mathfrak{T}(u)$, we only need to focus on the second one.  Besides, regarding this sum, it actually starts at $n=3$ because when $n=2$, the term is zero.

Then,  note that
\begin{equation}\label{eq;derivepartiel}
\partial_u \mathfrak T(u)=\sum_{n\geq 0} T_{n+1}\frac{u^{n}}{n!}.
\end{equation}
It comes  
\begin{equation}
\partial_u \mathfrak T(u)\mathfrak T(u)=
\sum_{n\geq 1}\left(\sum_{k=0}^{n-1}\frac{T_{k+1}T_{n-k}}{k!(n-k)!}\right) u^n=
\sum_{n\geq 1}\left(\sum_{k=1}^{n}\binom{n}{k-1}T_{k}T_{n-k+1}\right) \frac{u^n}{n!},
\end{equation}
and thus
\begin{eqnarray}
\int_0^u\partial_v\mathfrak T(v)\mathfrak T(v)dv 
& = &  \sum_{n\geq 1}\left(\sum_{k=1}^{n}\binom{n}{k-1}T_{k}T_{n-k+1}\right) \frac{u^{n+1}}{(n+1)!}\\
& = & \sum_{n\geq 3}\left(\sum_{k=1}^{n-2}\binom{n-2}{k-1}T_{k}T_{n-1-k}\right) \frac{u^{n-1}}{(n-1)!}. 	
\end{eqnarray}
This completes the proof.
\end{proof}

The differential equation \eqref{eq:ODE} then translates into a separable ODE:
\begin{equation}\label{eq:ODE2}
\partial_u \mathfrak T(u)=1+s\mathfrak T(u)+\dfrac{s}{2} \mathfrak T^2(u).
\end{equation}
Thereafter, we easily deduce that 
\begin{equation}\label{eq:int}
u=\int_0^{\mathfrak T(u)}\frac{dt}{1+st+s\frac{t^2}{2}}.
\end{equation}

\begin{lem}\label{lem:solveode}
If $s\in (0,2)$ then   
\begin{equation}
\mathfrak T(u)=\frac{2}{s}\times \frac{\tan\left(\frac{u\sqrt{s(2-s)}}{2}\right)}
{\sqrt{\frac{2-s}{s}}-   \tan\left(\frac{u\sqrt{s(2-s)}}{2}\right)}.
\end{equation}
If $s\in (2,\infty)$ then   
\begin{equation}
\mathfrak T(u)=\frac{2}{s}\times \frac{\sqrt{s(s-2)} e^{u\sqrt{s(s-2)}}}{s-1+\sqrt{s(s-2)}-e^{u\sqrt{s(s-2)}}}.
\end{equation}
\end{lem}

\begin{proof}[Proof of Lemma \ref{lem:solveode}]
The proof consists in standard computations of the integral \eqref{eq:int} noting that the discriminant of the quadratic term is given by $\Delta=s(s-2)$. 
\end{proof}

Let us set   
\begin{equation}\label{eq:singuexp}
u^\ast =
\begin{cases}
\frac{2}{\sqrt{s(2-s)}}\arctan\left(\sqrt{\frac{2-s}{s}}\right), & \text{if } s \in (0,2), \\[10pt]
\frac{1}{\sqrt{s(s-2)}}\log\left(s-1+\sqrt{s(s-2)}\right), & \text{if } s \in (2,\infty).
\end{cases}
\end{equation}

In each of these two cases, one can see that $u^\ast$ is the smallest  singularity of $\mathfrak T(u)$, since the other possible one is $\pi/\sqrt{s(2-s)}$ when $s\in(0,2)$.

\begin{rem}
The two formulas in \eqref{eq:singuexp} can be unified using the equality 
\begin{equation}
\arctan z=\frac{i}{2}\log\left(\frac{1-iz}{1+iz}\right), \quad z\in\mathbb C\setminus\{\pm i\}.
\end{equation}	
\end{rem}

\begin{lem}\label{lem:singularity}
For any $s\in(0,2)\sqcup(2,\infty)$ one has 
\begin{equation}\label{eq:asymsingular}
\mathfrak T(u)\underset{u\to u^\ast}{\sim}\frac{2}{s}\frac{1}{u^\ast-u}\quad\mbox{and}\quad \partial_u\mathfrak T(u)\underset{u\to u^\ast}{\sim}\frac{2}{s}\frac{1}{(u^\ast-u)^2}.
\end{equation}	
\end{lem}

\begin{proof}[Proof of Lemma \ref{lem:singularity}]
The proof follows from standard computations.
\end{proof}

To derive the asymptotics of $T_n$, we shall apply the transfer theorem as in the proof of Proposition \ref{prop:pascalweight}.  

To this end, note that $\tan(z)$ is analytic on $\{z \in \mathbb{C} : -\pi/2 < {\rm Re}(z) < \pi/2\}$, as is $r - \tan(z)$ for any $r > 0$. Set $\theta = \arctan(r) \in (0,\pi/2)$. There exists a $\mathbb{C}$-neighborhood $\mathcal{V}$ of $\theta$ such that $\tan$ induces a diffeomorphism from $\mathcal{V}$ onto its image, in such a way that for all $z \in \mathcal{V} \setminus \{\theta\}$, one has $r - \tan(z) \neq 0$.  

All of this shows that if $s \in (0,2)$, then $\mathfrak{T}$ admits an analytic continuation in some pacman domain of $\mathbb{C}$, where the corner at the mouth coincides with the singularity $u^\ast$.  

Similar arguments also show that this remains true when $s \in (2,\infty)$.  

As a consequence, we obtain that for all $s \in (0,2) \sqcup (2,\infty)$,  
\begin{equation}\label{eq:transfer}
T_n \sim \frac{2}{s} \frac{n!}{(u^\ast)^{n+1}}.
\end{equation}

\begin{rem}
When $s = x + y = 2$, we obtain $\mathfrak{T}(u) = \frac{u}{1 - u}$, which is consistent with \eqref{eq:transfer}, since in that case, one can easily show that $T_n = n!$ (as when $x = y = 1$).  

Besides, it is not entirely clear that $T_n$ increases with $s$, as it should. However, one can show that the function $s \mapsto u^\ast$ is extendable by continuity at $s = 2$ by setting $u^\ast = 1$ when $s = 2$ and verify that $s \mapsto s \times (u^\ast)^{n+1}$ is increasing on $(0,\infty)$ for all $n\geq 1$.  
\end{rem}

To find the MERWs, we need to compute the asymptotics of ${d(\tau, \mathbb{X}_{n})}/{d(\varnothing, \mathbb{X}_{n})}$  
for all finite binary trees $\tau$. Any limit point $\varphi(\tau)$ will produce an NNHF associated with a MERW according to the definition of such random walks.  

Let us introduce  
\begin{equation}\label{eq:leafnotcomplete}
\ell(\tau) = \#\partial \tau \quad \text{and} \quad \kappa(\tau) = \#\{i \in \overset{\circ}{\tau} : i1 \notin \tau \text{ or } i2 \notin \tau\},
\end{equation}
respectively, the number of leaves and the number of incomplete internal nodes of $\tau$.  Note that  
\begin{equation}\label{eq:numberleave}
2\ell(\tau) + \kappa(\tau) = |\tau| + 1 = \ell(\tau) + |\overset{\circ}{\tau}| + 1.
\end{equation}
These equalities will be useful later.

\begin{lem}\label{lem:combiasymptree}
For any $n \geq k$, we claim that  
\begin{equation}\label{eq:producT}
d(\tau, \mathbb{X}_{n-1}) = \sum_{\substack{i_1+\cdots+i_\ell+ \\ j_1+\cdots+j_\kappa=n-k}}  
\binom{n-k}{i_1, \dots, i_\ell, j_1, \dots, j_\kappa}  
T_{i_1+1} \cdots T_{i_\ell+1} T_{j_1}^{\star} \cdots T_{j_\kappa}^\star,
\end{equation}
where $\ell := \ell(\tau)$, $\kappa := \kappa(\tau)$, the indices $i_1, \dots,i_\ell, j_1, \dots,j_\kappa$ are non-negative integers, and we define $T_j^\star = T_j$ for all $j \geq 1$, with $T_0^\star = 1$. 

In particular, we obtain  
\begin{equation}\label{eq:producT2}
d(\tau, \mathbb{X}_{n-1}) = \left[\frac{u^{n-k}}{(n-k)!}\right]
\{\partial_u \mathfrak{T}(u)\}^{\ell(\tau)} \{\mathfrak{T}(u)+1\}^{\kappa(\tau)},
\end{equation} 	
where $[Z]F$ denotes the coefficient of $Z$ in $F$, interpreted as a formal series as usual.
\end{lem}

\begin{proof}[Proof of Lemma \ref{lem:combiasymptree}]
To begin with, we make the following key observations:
\begin{itemize}
\item $T_{i+1}$ counts the total weight of increasing continuations of size $i$ for a rooted binary tree with only one element (since the root necessarily has the smallest label) and 
\begin{equation}\label{eq:1}
T_{i+1}=\left[\frac{u^i}{i!}\right]\partial_u\mathfrak T(u).
\end{equation}
\item $T^\star_j$ counts the total weight of increasing continuations of size $j$ for a rooted binary tree with only one element and whose root has only one fixed child (the left or the right) and
\begin{equation}\label{eq:2}
T_{j}^\star=\left[\frac{u^j}{j!}\right](\mathfrak T(u)+1).
\end{equation} 
\end{itemize}

\begin{rem}
	Note that $T_{0+1} = T^\star_0 = 1$, in accordance with Definition \ref{def:combidim}: the weight of a continuation of size zero is equal to one.
\end{rem}

Therefore, to compute the weighted total number of paths from $\tau$ to an arbitrary $\sigma \supset \tau$ of size $n$ (thus belonging to $\mathbb{X}_{n-1}$), one can proceed as follows:  
\begin{enumerate}
	\item Choose the number of descendants $i_1, \dots, i_\ell$ and $j_1, \dots, j_\kappa$ (excluding the nodes themselves) for each leaf $v_1, \dots, v_\ell$ and each incomplete internal node $w_1, \dots, w_\kappa$;
	\item Distribute the $n-k$ labels into $\kappa+\ell$ groups of sizes $i_1, \dots, i_\ell, j_1, \dots, j_\kappa$ associated with the nodes $v_1, \dots, v_\ell, w_1, \dots, w_\kappa$;
	\item Sum over all possibilities the corresponding total weight of such increasing continuations, which in any case is given by $T_{i_1+1} \cdots T_{i_\ell+1} T_{j_1}^\star \cdots T_{j_\kappa}^\star$.
\end{enumerate}

Based on this counting principle, we deduce \eqref{eq:producT}, and equality \eqref{eq:producT2} follows easily from \eqref{eq:1} and \eqref{eq:2}. This completes the proof of the lemma.
\end{proof}

Finally, by using \eqref{eq:asymsingular}, we obtain
\begin{equation}
\{\partial_u \mathfrak{T}(u)\}^{\ell(\tau)} \{\mathfrak{T}(u)+1\}^{\kappa(\tau)}\underset{u\to u^\ast}\sim \left(\frac{2}{s}\right)^{\ell(\tau)+\kappa(\tau)}\frac{1}{(u^\ast-u)^{k+1}}.
\end{equation}
Here we use that $2\ell(\tau)+\kappa(\tau)=k+1$.

Applying again the transfer theorem as for \eqref{eq:transfer}, we deduce that
\begin{equation}
d(\tau, \mathbb{X}_{n-1})\underset{n\to\infty}{\sim} \left(\frac{2}{s}\right)^{\ell(\tau)+\kappa(\tau)}\frac{n^k}{(u^\ast)^{n+1}}\frac{(n-k)!}{k!},
\end{equation}
and thus 
\begin{equation}
\lim_{n\to\infty}\frac{d(\tau, \mathbb{X}_{n-1})}{d(\varnothing, \mathbb{X}_{n-1})}=\frac{1}{k!} \left(\frac{2}{s}\right)^{\ell(\tau)+\kappa(\tau)-1}.
\end{equation}
The proof of the expression of the PHF is then straightforward by using \eqref{eq:leafnotcomplete}, as the above equation yields the transition probabilities of the corresponding MERW.

It remains to show the asymptotics \eqref{eq:asympBSTgen}. To this end, we shall write the PHF \eqref{eq:harmmerwbstgen} as the expectation of a random PHF, as in the BST process.  

We adopt the fragmentation measure approach to express the general NNHF \eqref{harmpref} as  
\begin{equation}
\phi(\tau) = \prod_{i\in\overset{\circ}{\tau}} V_{i,1}^{|\tau^{(i1)}|} V_{i,2}^{|\tau^{(i2)}|}  
\prod_{i\in\overset{\circ}{\tau}} \frac{1}{x V_{i,1} + y V_{i,2}},
\end{equation}
where, for the moment, the pairs $(V_{i,1},V_{i,2})$, with $i\in\mathcal{S}_2$, are arbitrary parameters belonging to the 1-dimensional simplex $\{(t,1-t) : 0\leq t\leq 1\}$.  

Note that the first product in the latter equality corresponds to the expression \eqref{diri}, which holds in the unweighted case $x = y = 1$.  

Now, choosing the family $(V_{i,1},V_{i,2})_{i\in\mathcal{S}_2}$ to be \emph{i.i.d.\@} and distributed as \eqref{eq:density} over the simplex, it follows that  
\begin{equation}
\mathbb{E}[\phi(\tau)] = \prod_{i\in\overset{\circ}{\tau}}  
\mathbb{E} \left[\frac{V_{i,1}^{|\tau^{(i1)}|}V_{i,2}^{|\tau^{(i2)}|}}{x V_{i,1} + y V_{i,2}}\right]  
= \left(\frac{2}{x+y}\right)^{|\overset{\circ}{\tau}|}  
\prod_{i\in\overset{\circ}{\tau}} \mathbf{B}(|\tau^{(i1)}|+1,|\tau^{(i2)}|+1) = \varphi(\tau).
\end{equation}
Recall that the computation of the product involving the beta functions in the latter equality has already been done (see equations \eqref{mean} and \eqref{transidla}).  

Since the sequence $(\alpha_i)_{i\in\mathcal{S}_2}$, which describes the NNHFs in \eqref{harmpref}, still represents the almost sure asymptotic proportion of descendants of $i \in \mathcal{S}_2$ in $\tau_n$ (as in Theorem \ref{bratstructure}), and since $\alpha_{i,j} = V_{i,j} \alpha$ for any $i \in \mathcal{S}_2$ and $j \in \{1,2\}$, we obtain the result as in Corollary \ref{coro:BST}.  

This completes the proof.
\end{proof}

\begin{rem}
There are endless possibilities for more general weighting schemes for trees, such as, for example, the case when the weights depend on behaviour of the nodes at a fixed distance from the main node. We can also consider a situation where the weight assigned to adding the $j$th child of a node $i$ depends on the presence of other children that have already been added. The formulas become quite heavy, but no conceptual difference is present.	
\end{rem}

\section{Infinite Combs and Agregation Processes}

\label{sec:combagreg}

In this section, we show that BD can be very convenient for studying aggregation processes, especially those related to growing tree structures, as in Section \ref{section:tree}.  

The first model we introduce is indeed an example of a tree growth process, whose skeleton is an infinite comb. It turns out that the MERW is trivial, but this model is still rich.  

It is related to random partitions (Bell numbers), and the combinatorial formula (\ref{combipower}) produces, as an example, a well-known identity involving Stirling numbers \eqref{stirling}.  

Moreover, choosing a suitable random environment leads to the famous Chinese restaurant process and allows us to retrieve well-known asymptotics.  

Finally, we discuss extensions of this model, some still related to tree growth processes, others much more difficult to handle and investigate, laying the groundwork for the last section on computer simulations.

\subsection{Description of the Aggregation Model and its BD structure}  

\label{sec:agregmodel}

We consider an aggregation process where individuals sequentially form groups according to simple rules.  

The first individual, say $I_1$, arrives and creates a group of size one, denoted by $(1)$. When the second individual arrives, they have two choices:  
1) Join the existing group created by $I_1$, forming a group of size two, denoted as $(2)$.  
2) Create a new group, resulting in two separate groups of size one, denoted as $(1,1)$.  More generally, when the $n$th individual arrives, they can either:  
1) Join any of the existing groups, increasing its size.  
2) Create a new group, leading to a different partition of individuals.

This iterative process generates a sequence of group formations, where the structure depends on the aggregation rules chosen. This type of dynamics can model the evolution of entities such as political parties or consumer choices.\\

\noindent
\textit{i) The BD structure.}  The $n$th level set of such process can be described by  
\begin{equation}
\mathbb{X}_{n} = \{(x_1, \dots, x_k) : k \geq 1, \; x_1, \dots, x_k \geq 1, \; x_1 + \dots + x_k = n+1\}.
\end{equation}
Note that $\mathbb{X}_0 = \{(1)\}$ and, more generally, the $n$th level consists of elements of size $n+1$.  

Then, denote by $\ell(x)$ the integer $k$ -- the number of groups -- appearing in the latter statement. One has $x \nearrow y$ in the following two cases:
\begin{enumerate}
	\item[i)] $\ell(x) = \ell(y)$ and $y_i = x_i + 1$ for some $i$, with $y_j = x_j$ for all $j \neq i$;
	\item[ii)] $\ell(y) = \ell(x) + 1$, $y_j = x_j$ for all $1 \leq j \leq \ell(x)$, and $y_{\ell(y)} = 1$.
\end{enumerate}

Similarly to the Young lattice, a state $x \in \mathbb{X}$ can be visualized as a stack of boxes (unit squares) with its bottom-left corner at coordinates $(i,j) \in \mathbb{N}^2$. Paths within $\mathbb{X}$ are thus represented by connected diagrams, with an enumeration of the boxes.

\begin{figure}[H]
	\centering
	\includegraphics[width=0.55\textwidth]{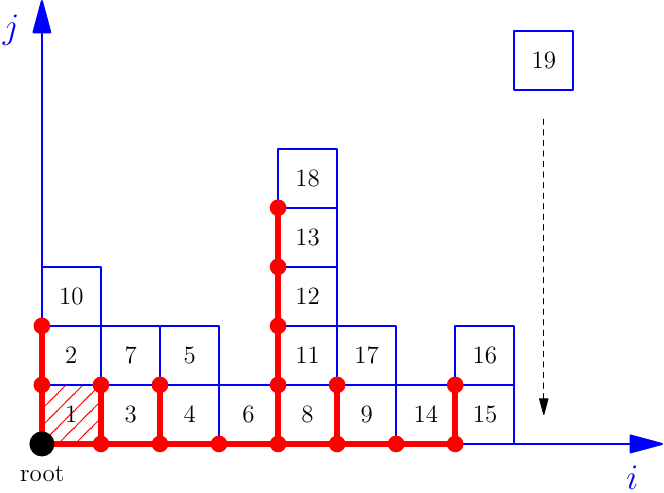}
	\captionsetup{width=0.8\textwidth}
	\caption{A path in $\mathbb{X}$ from the root to $x = (3,2,2,1,5,2,1,2)$ of size 18, with its corresponding tree structure.}
	\label{tower}
\end{figure}

\noindent
\textit{ii) Underlying tree structure.} This model is similar to a growing tree model (see Figure \ref{tower}). Indeed, let $\mathcal{C}$ be the infinite comb tree defined as $\mathcal{C} = \{2^i1^j : i, j \geq 0\}$. 

Here, $2^i1^j$ (with $2$ repeated $i$ times and $1$ repeated $j$ times) represents a node in $\mathcal{C}$. The root, corresponding to $i=j=0$, is still denoted by $\{\emptyset\}$ as in Section \ref{section:tree}.

To each $x \in \mathbb{X}_n$, it corresponds a finite comb, which is a finite prefix tree, given by   
\begin{equation}
\tau(x) = \{2^i1^j : 0 \leq i \leq \ell(x)-1, \, 0 \leq j \leq x_j-1 \} \subset \mathcal{C}.
\end{equation}
Note that $\tau((1))=\{\emptyset\}$ the tree with one element, the root. Reciprocally, to each finite prefix tree $\tau\subset\mathcal C$, there exists a unique $x\in\mathbb X$ such that $\tau(x)=\tau$.\\

\noindent
\textit{ii) The Martin boundary.} Based on this correspondence, we can apply Theorem \ref{bratstructure} and we deduce that the Martin boundary is given by  
\begin{equation}
\partial\mathbb{X} = \bigg\{(\theta_k)_{k\geq 1} \in [0,1]^{\mathbb{N}} : \sum_{k \geq 1} \theta_k \leq 1 \bigg\},
\end{equation}
where the $\theta_k$ represent the asymptotic behavior of $x_k/n$, where $x_k$ is the $k$th coordinate of an element $x \in \mathbb{X}_n$, as $n$ tends to infinity. Note also that $\theta_k$ represents the asymptotic proportion of boxes in the $k$-th tower in the diagram represented in Figure \ref{tower}. Besides, the corresponding boundary point $(\alpha_i)_{i\in \mathcal{C}}$ of the growing tree model is represented in Figure \ref{fig:comb}.

In addition, the corresponding extremal NNHFs  are given by
\begin{equation*}
\varphi_\theta(x_1, \cdots, x_k) = \theta_1^{x_1 - 1} \cdots \theta_k^{x_k - 1} (1 - \theta_1) \cdots \left(1 - \sum_{j=1}^{k-1} \theta_j\right).
\end{equation*}

Regarding the corresponding ergodic saturated central Markov chain, the probability to join the $k$-th group, if it exists, is equal to $\theta_k$, and the probability to create a new one when $n$ groups already exist is equal to $1 - \sum_{j=1}^n \theta_j$. 
	
\begin{figure}[H]
\centering
\includegraphics[width=0.65\textwidth]{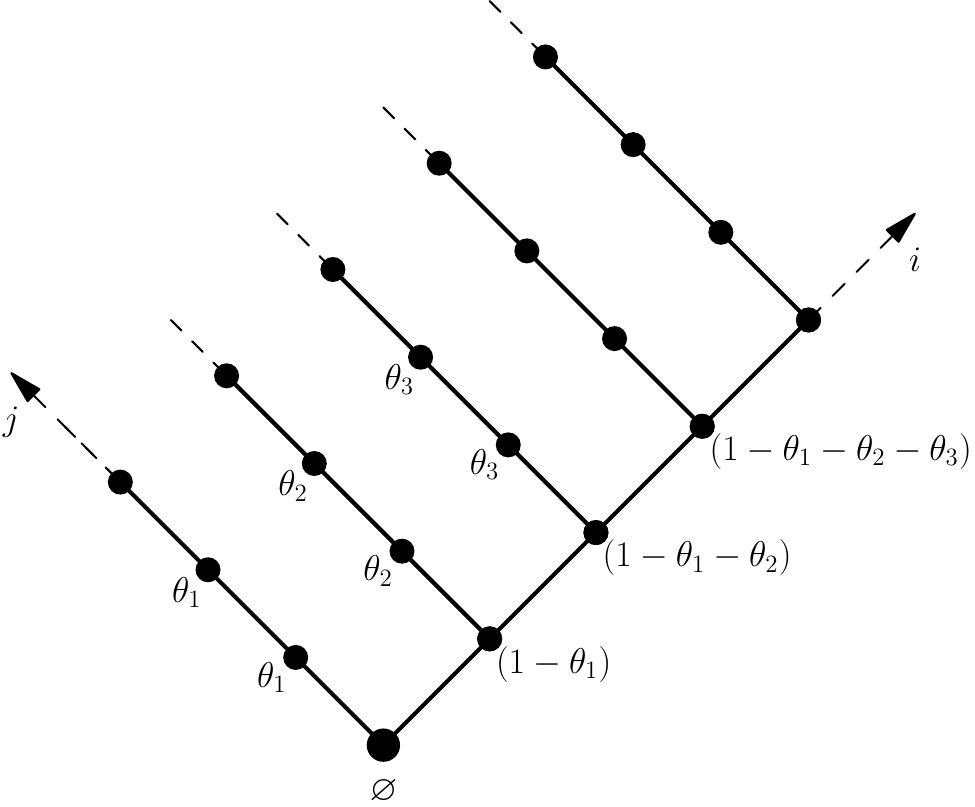}
\captionsetup{width=0.8\textwidth}
\caption{The Martin Boundary associated with the comb-skeleton $\mathcal {C}$.}
\label{fig:comb}
\end{figure}

\subsection{Uniform random partitions and the MERW} 

There exists a one-to-one correspondence between paths of length $n$ starting from the root and partitions of the set $\{1, \cdots, n\}$. This correspondence is illustrated in Figure \ref{bell}.

\begin{figure}[H]
	\centering
	\includegraphics[width=0.7\textwidth]{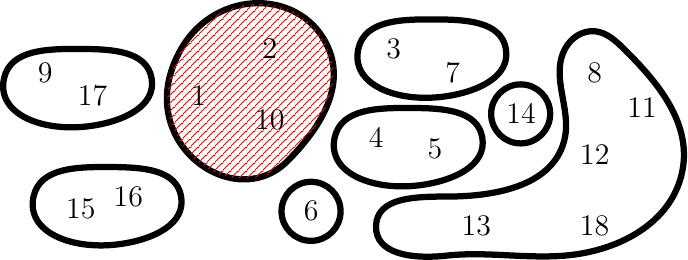}
	\captionsetup{width=0.8\textwidth}
	\caption{The partition corresponding to the path given in Figure \ref{tower}.}
	\label{bell}
\end{figure}

As a consequence, the combinatorial dimension $d(\varnothing, \mathbb{X}_{n-1})$ is equal to the Bell number $B_n$. Hence, the uniform distribution over the set of paths $\mathcal{T}_n$, denoted by $m_n$ as in (\ref{marginalunif}), corresponds to the distribution of a uniform random partition of $\{1, \cdots, n\}$. 

This distribution has been extensively studied in the literature, for instance, in \cite{Stam,Pitman}. Notably, it has been demonstrated that the asymptotic number of sets is of order $n/\ln(n)$, each set having approximately a cardinal equal to $\ln(n)$. 

We deduce the following result in the same manner as Proposition \ref{plancherelgrowth} for the Plancherel growth process.

\begin{prop}\label{trivial}
	The unique MERW associated with the agregation process in Section \ref{sec:agregmodel} is the deterministic process corresponding to boundary point $\theta = 0$.
\end{prop}

\begin{rem}
	If the number of groups is limited by $N$, it is not difficult to see that the unique MERW corresponds to $\theta_k = {1}/{N}$ for all $1 \leq k \leq N$ and $\theta_k = 0$ for all $k > N$.
\end{rem}

\begin{rem}
	Let $S(n,k)$ be the Stirling number of the second kind, which represents the number of ways to partition $\{1, \cdots, n\}$ into $k$ subsets. 
	
	Then by setting $\theta_i \equiv X^{-1}$ and using equation (\ref{comb2}), we can derive and reinterpret the well-known combinatorial identity:
	\begin{equation}\label{stirling}
	\sum_{k=1}^n S(n,k) \prod_{j=1}^{k-1}(X-j) = X^n.
	\end{equation}	
\end{rem}

\subsection{The Kingman's law and the Chinese restaurant process}  

Similarly to Section \ref{sec:randomenv}, consider a random point of the Martin boundary, given here by  $\theta_1=U_1$ and, for all $n\geq 1$,
\begin{equation}
\theta_{n+1}=U_{n+1}(1-\theta_1-\cdots-\theta_n),
\end{equation}
where $(U_n)_{n\geq 1}$ are \emph{i.i.d.\@} beta random variables distributed as $\beta(1,\gamma)$. Note also that 
$$\theta_n=U_n(1-U_{n-1})\cdots(1-U_1).$$
    
It is well-known that the order statistics $\theta_{(1)}>\theta_{(2)}>\cdots$ follow the so-called  Kingman's law with parameter $\gamma$, corresponding to the Poisson-Dirichlet distribution ${\rm PD}(0,\gamma)$. We refer to \cite{Pitman2} for more details. 

To derive the annealed NNHF $\overline \varphi$, similarly to that in (\ref{mean}), we can observe that $(\theta_1,\cdots,\theta_k)$ has for density on $(0,1)^k$ the function 
\begin{equation}
f(s_1,\cdots,s_k)=\frac{1}{\gamma^k}\frac{(1-s_1-\cdots-s_k)^{\gamma-1}}{(1-s_1)\cdots(1-s_1-\cdots-s_{k-1})}.
\end{equation}

Thereafter, one can check that $\overline \varphi$ can be obtained by computing 
\begin{equation}
\frac{1}{\gamma^k}\int_{(0,1)^k} s_1^{x_1-1}\cdots s_k^{x_k-1}(1-s_1-\cdots-s_k)^{\gamma-1} ds_1\cdots ds_k.
\end{equation}
This integral is related to the Generalized Dirichlet distribution in \cite{Wong}. We get 
$$\overline \varphi(x_1,\cdots,x_k)=\frac{1}{\gamma^k}\prod_{i=1}^k \beta(x_i,x_{i+1}+\cdots+x_k+\gamma).$$

Assuming that $x_1+\cdots+x_k=n$, the transition probabilities of the corresponding central Markov chain are given by 
\begin{equation}
\overline p((x_1,\cdots,x_k);(x_1,\cdots,x_i+1,\cdots,x_k))=\frac{x_i}{n+\gamma},
\end{equation}
and
\begin{equation}
\overline p((x_1,\cdots,x_k);(x_1,\cdots,x_k,1))=\frac{\gamma}{n+\gamma}.
\end{equation} 
This corresponds to the Chinese restaurant process, as detailed in \cite[p. 92]{Aldous} and \cite{Pitman2}. 

Again, and surprisingly, Theorem \ref{convcentral} and the above computation allow us to derive its asymptotics in a simpler way, independently of the theory of exchangeable partitions.

\begin{coro}\label{chinese}
	Let $(X_1(n), X_2(n), \cdots)_{n\geq 0}$ be the central Markov chain corresponding the the Chinese restaurant process. Consider $X_{(1)}(n)\geq X_{(2)}(n)\geq \cdots$ the order statistics. 
	
	Then, we have the following convergence in distribution as $n$ tends to infinity:
	$$ \left( X_{(i)}(n) \right)_{i\geq 1} \;\underset{n\to\infty}{\Longrightarrow}\; {\rm PD}(0,\gamma). $$
\end{coro}

\subsection{A simple extension} 

A natural extension of the model in Section \ref{sec:agregmodel} is to allow boxes to be placed to the left of an existing box, while still maintaining the tree growth model structure. We refer to Figure \ref{bidim} for an illustration.

One can easily show that the Martin boundary in this case is given by
\begin{equation}
\left\{\alpha, \beta \in [0,1], (\theta_k)_{k\in\mathbb{Z}} \in [0,1]^{\mathbb{Z}} : \sum_{k=1}^\infty \theta_k \leq \alpha, \sum_{k=1}^\infty \theta_{-k} \leq \beta, \alpha + \beta + \theta_0 = 1 \right\}.
\end{equation}
Here, $\theta_k$ represents the asymptotic proportion of squares with their bottom-left corner at $(k,j)$ for some $j$. The parameters $\alpha$ (resp. $\beta$) represent the proportion of boxes with their bottom-left corner at $(i,j)$, where $i\geq 1$ (resp. $i\leq -1$). 

Additionally, the corresponding extremal harmonic function is given by
\begin{equation}
\varphi_{(\theta,\alpha,\beta)}((x_{-p},\cdots,x_0,\cdots,x_q))=\alpha^q\beta^p\prod_{k=-p}^q \theta_k^{x_k-1}\prod_{k=1}^{q-1}\left(1-\frac{\theta_k}{\alpha}\right)\prod_{k=1}^{p-1}\left(1-\frac{\theta_{-k}}{\beta}\right).
\end{equation}

Taking into account symmetry and Proposition \ref{trivial}, one can easily verify that the only MERW is characterized by $\theta\equiv 0$ and $\alpha = \beta = {1}/{2}$. It can be viewed as a simple symmetric random walk on $\mathbb{Z}$.

\begin{figure}[H]
	\centering
	\includegraphics[width=0.6\textwidth]{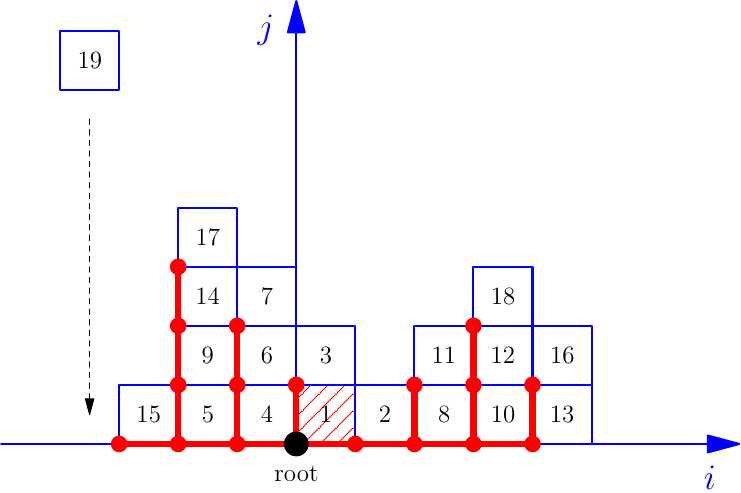}
	\captionsetup{width=0.8\textwidth}
	\caption{A path  and its underlying tree structure.}
	\label{bidim}
\end{figure}


\subsection{Growing pyramidal diagrams}

In order to obtain non-trivial MERWs, a particularly challenging problem is to incorporate constraints on the shape within the aggregation model illustrated in Figure \ref{bidim}.

Consider the following variant: let $h_i$, for $i \in \mathbb{Z}$, denote the number of squares whose bottom-left corners are located at $(i,j)$. In cases where $h_k \geq 1$, a box may be placed above, to the right, or to the left of the $k$th tower, but this is subject to the restriction that the resulting configuration does not produce a scenario where $h_k < \min(h_{k-1}, h_{k+1})$.  

In other words, only configurations that maintain a pyramidal shape are permitted. Figure \ref{pyra} provides an example of such a configuration.

\begin{figure}[H]
	\centering
	\includegraphics[width=0.45\textwidth]{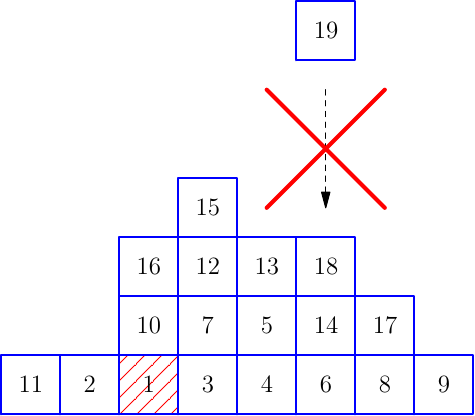}
	\captionsetup{width=0.8\textwidth}
	\caption{An example of a path of pyramidal shapes.}
	\label{pyra}
\end{figure}

It is important to note that the BD associated with this variant does not correspond to a tree growth model (unfortunately). 

To assess the level of difficulty, consider a scenario where the bases of the pyramids are restricted to a maximum size of three. Denote by $l$, $c$, and $r$ the three possible choices for the abscissa at each step. 

This model, which seems much simpler than the previous one, is linked (in the sense of Section \ref{section:connexion}) to the Kreweras's random walk in the three-quarter plane.

This walk occurs on $\mathbb{Z}^2 \setminus \{(i,j) : i,j < 0\}$ with steps in $\{\text{North-East; South; West}\}$. We refer to Figure \ref{krew} for an illustration.

\begin{figure}[H]
	\centering
	\includegraphics[width=0.7\textwidth]{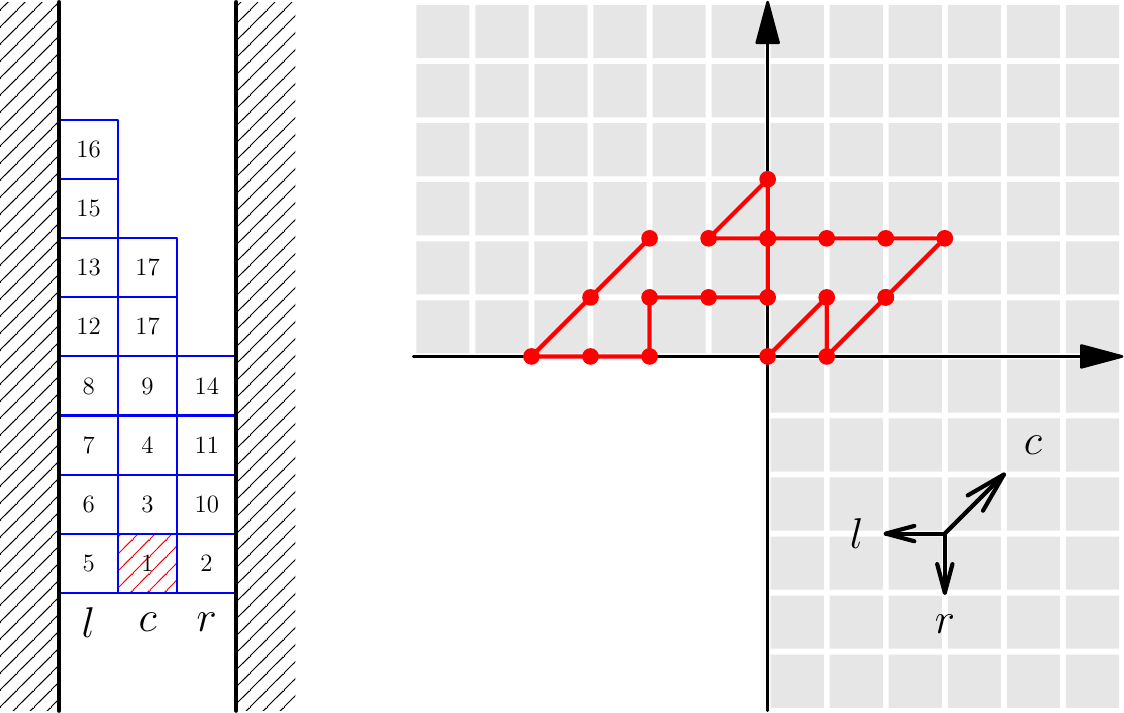}
	\captionsetup{width=0.8\textwidth}
	\caption{Correspondence between the Kreweras's random walk and a paths in the BD of pyramidal diagrams.}
	\label{krew}
\end{figure}

The enumeration of walks with small steps in cones has been a difficult combinatorial task, often requiring powerful algebraic methods. The generating function for the total number of walks originating from the origin is detailed in \cite{bousquet}, yielding the asymptotic 
\begin{equation}
a_n(0,0) \sim \frac{3^{3/4} \sqrt{2-\sqrt{2}}}{\Gamma(5/8)} \frac{1}{n^{3/8}} 3^n.
\end{equation}

To find central Markov chains and thus MERWs, we need to investigate the limit of the ratio 
\begin{equation}
\frac{a_{n-k}(i,j)}{a_n(0,0)},
\end{equation}
 as $n \to \infty$, where $a_m(i,j)$ represents the total number of walks of length $m$ starting from $(i,j)$. 
 
 Unfortunately, the generating functions in \cite{bousquet} focus on $c_n(i,j)$, the number of walks of length $n$ starting from the origin and ending at $(i,j)$, which is not directly applicable here. Nonetheless, it is reasonable to conjecture that $a_n(i,j)$ adheres to the same asymptotic as $a_n(0,0)$, modulo a constant $H_{i,j}$. Consequently, $H$ can be identified as a unique positive harmonic function for Kreweras's random walk in the three-quarter plane, with $H_{0,0} = 1$, as established by \cite{Trotignon}. 
 
 We make the following conjecture.
\begin{conj}
	The unique MERW for the growing pyramidal model depicted in Figure \ref{krew} is characterized by the following transition probabilities. Given $(u,v,w)$ as a permissible configuration derived from $(x,y,z)$ by incrementing one of its coordinates by $1$:
	\begin{equation}
	p((x,y,z);(u,v,w)) = \frac{H_{v-u,v-w}}{3\, H_{y-x,y-z}},
	\end{equation}
where $H$ is the unique non-negative solution on $\mathbb Z^2$ of  $H_{i,j}=0$ for $i,j< 0$, $H_{0,0}=1$ and for all $i\geq 0$ or $j\geq 0$, 
$3 H_{i,j}=H_{i-1,j}+H_{i,j-1}+H_{i,j}$.
\end{conj}

\section{Numerical Simulations}

\label{sec:algo}
\setcounter{equation}{0}

In general, explicitly computing MERWs or central Markov chains is almost miraculous.  

However, one might wonder whether efficient numerical simulations can be performed to estimate the combinatorial dimension in (\ref{finiteMERW}) and approximate the transition probabilities.  

Obviously, recursively counting the number of paths (focusing only on the unweighted case) is generally unrealistic due to the exponential (or even worse) growth of such numbers.  

This is why we advocate a Monte Carlo (MC) method.

\subsection{The Algorithm}

We employ the Knuth's algorithm, as presented in \cite{Knuthtree}, to enumerate the leaves of a tree through MC simulations. 

Several adaptations of the original Knuth algorithm have been proposed, all aiming to reduce variance. For instance, we refer to \cite{Jensen1, Jensen2} and \cite{Cloteaux}. 

Although a BD $\mathbb{X}$ is not necessarily a tree, the set $\mathfrak R = \bigsqcup_{n \geq 0} \mathcal{T}_n$ is. We recall that $\mathcal{T}_n$ denotes the set of finite paths of length $n$ originating from the root, as introduced in Section \ref{sec:WBD}. 

A finite path $t$ is a child of another path $s$ if and only if $t$ can be obtained from $s$ by adding an additional transition at the end of $s$. 

Thus, computing $d(x, \mathbb{X}_n)$ is equivalent to enumerating the number of leaves of the finite subtree of paths of length $d = n - n_x$ starting from $x$. 

To this end, Knuth proposes to sample $d$-step trajectories $s = s_0 \cdots s_d$ of a given RW -- say having $q$ for Markov kernel --  starting from $x$, and to compute the mean of the cost function
\begin{equation}\label{Knuth}
c(s) = (q(s_0, s_1) \cdots q(s_{d-1}, s_d))^{-1},
\end{equation}
This leads to Algorithm \ref{Algo} written in pseudocode.

\begin{algorithm}[H]
	\caption{Approximation of a MERW Based on Knuth's Algorithm}   
	\begin{algorithmic}
		\STATE \textbf{Initialize:} Choose a probability kernel $q(x,y)$ on $\mathbb{X}$, a depth of exploration $d$, a sample size $N$, and a path length $n$. Set $T_0 = \varnothing$ and $T_i = t_0 \cdots t_i$ as the current path at time $i \geq 0$.  
		\FOR{each step \(i\), where \(i < n\),}
		\FOR{each $x$ such that $t_i \nearrow x$}
		\STATE Generate $N$ trajectories  $s^{(1)}, \cdots, s^{(N)}$ of size $d$ of the RW $q$ starting from $x$. 
		\STATE Calculate (see \ref{Knuth}) the mean weight:
		\begin{equation}\label{estimate}
		W(x) = \frac{1}{N} \sum_{k=1}^N c(s^{(k)}).
		\end{equation}
		\ENDFOR
		\STATE Draw a random neighbour $y$ of $t_i$ with a probability proportional to $W(y)$.
		\STATE Set $t_{i+1} = y$.
		\ENDFOR
		\RETURN the trajectory \( t_0 \cdots t_n\).
	\end{algorithmic}
	\label{Algo}
\end{algorithm}

\subsection{Applications to Pyramidal Tableau}

Here we consider the pyramidal growing model illustrated in Figure \ref{pyra}. 

The initial approach is to utilize the generic random walk (GRW) in Algorithm \ref{Algo}. We denote this method as $(RW_0)$. However, the challenge in applying Knuth's algorithm effectively is to identify a random walk (RW) for which the standard deviations of the estimates (\ref{estimate}), expressed as $s/\sqrt{N}$, are not excessively large. 

Table \ref{tab:example} provides the theoretical number of $n$-step trajectories originating from the root and the corresponding Standard Deviations (StD) $s_0$ and $s_1$ for two methods: the initial one $(RW_0)$ and another, more efficient method $(RW_1)$, which will be detailed further below.

\begin{table}[H]
	\centering
	\begin{tabular}{|c|c|c|c|c|}
		\Xhline{2\arrayrulewidth}
		\textbf{Length $n$} & \textbf{Number of Paths} & \textbf{StD  ($s_1$)} & \textbf{StD ($s_0$)} \\
		\Xhline{4\arrayrulewidth}
		1 & 3 &  (0) & (0) \\
		\hline
		2 & 11 &  (0) & (1.418)\\
		\hline
		3 & 47 &  (4.501) & (11.42) \\
		\hline
		4 & 213 & (8.615) &  (76.59)\\
		\hline
				5 & 1 013 & (23.64) & (477.9) \\
		\hline
				6 & 5 047 & (261.8) & (2823) \\
		\hline
				7 & 26 077 & (2 569) &  (17 020)\\
		\hline
				8 & 143 067 & (21 120) &  (106 400)\\
		\hline
				9 & 809 973 & (158 600) &  (596 700) \\
		\hline
				10 & 4 758 653 & (1 161 000) &  (3 759 000)\\
				\hline
		11 & 28 892 669  & (8 341 000) &  (25 570 000) \\
				\hline
		12 & 180 970 405 & (59 460 000) & (172 900 000) \\
		\hline		
		13 & 1 166 654 573  & (425 600 000) & (1 257 000 000) \\		
\hline
	\end{tabular}

\captionsetup{width=0.8\textwidth}
\caption{Number of paths originating from the root and standard deviations of their Monte Carlo estimations for two different RW kernels.}
	\label{tab:example}
\end{table}

\noindent
{\it i) The Random Walk ($RW_1$).} Given a pyramidal diagram $\Delta$, we denote by $K(\Delta)$ the number of boxes $\square$ such that $\Delta \nearrow \Delta \cup \square$, representing the out-degree in the corresponding BD. 

For instance, $K(\Delta)=8$ for the pyramidal Tableau in Figure \ref{pyra}. The transition kernel $q$ considered in Algorithm \ref{Algo} takes the form:
\begin{equation}
q(\Delta,\Delta^\prime)=\frac{K(\Delta^\prime)^\gamma}{\sum_{\square} K(\Delta \cup \square)^\gamma}.
\end{equation}

We assume that $\Delta \nearrow \Delta^\prime$ and all the $\square$ in the sum are available boxes of $\Delta$. The parameter $\gamma$ may depend on the size $n$ of $\Delta$. In Table \ref{tab:example}, we adjust $\gamma$ such that $\gamma=2.5$ when $n=1$, $\gamma=1$ when $n=d-2$, and $\gamma$ decreases linearly between these two steps, reaching $\gamma=0$ when $n=d-1$. These parameters have been consistently applied in all numerical simulations.\\

\begin{figure}[H]
	\centering
	\includegraphics[width=1\textwidth]{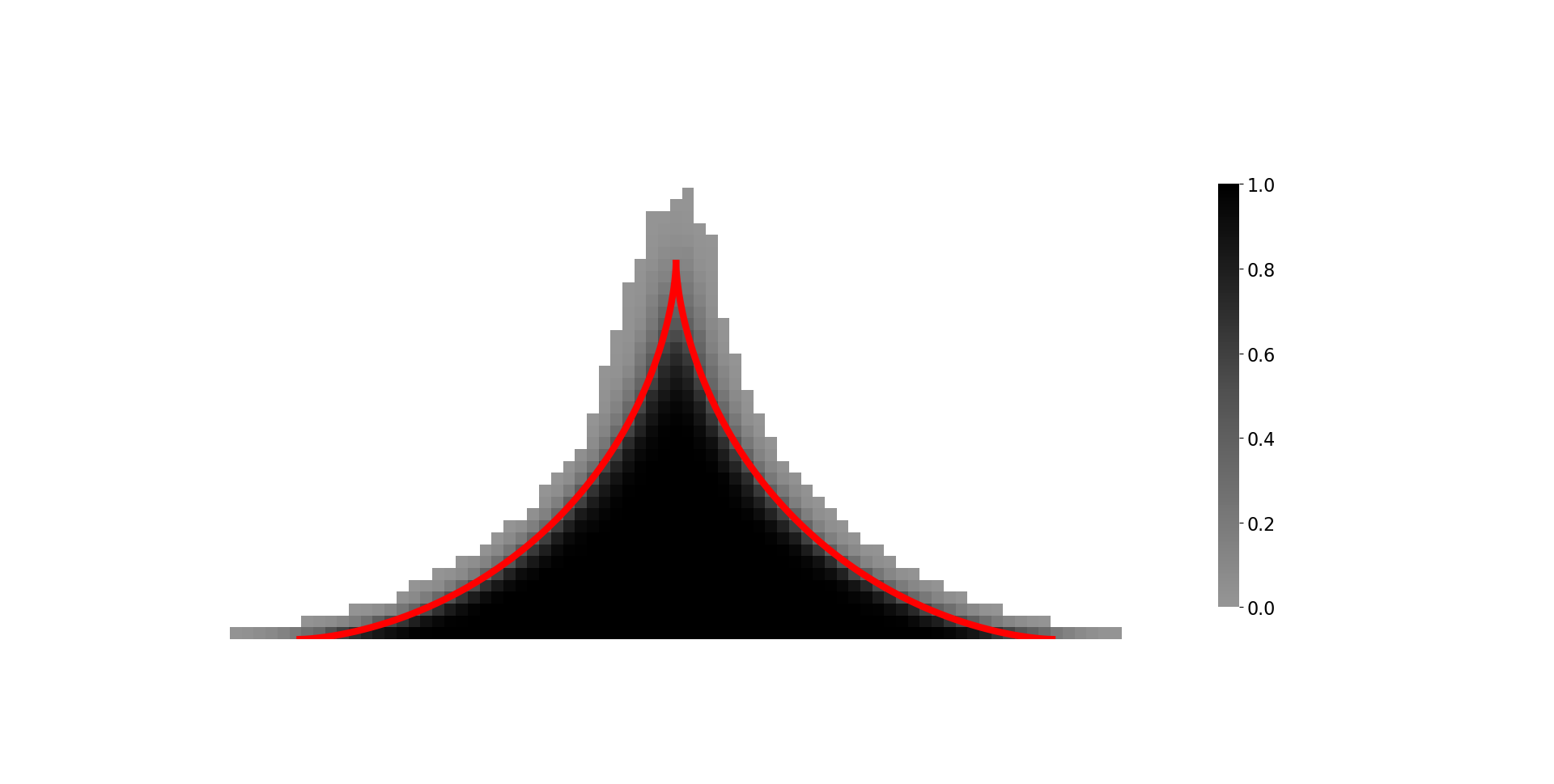}
	\captionsetup{width=0.8\textwidth}
	\caption{Heat map of approximations of the MERW Pyramidal process. Pyramid size: $n = 500$. Number of pyramid samples: $N_P = 500$. Parameters of Algorithm \ref{Algo}: $N = 100$, $d = 30$.}
	\label{MeanMERW}
\end{figure}

\noindent
{\it ii) Numerical Simulations and Conjecture.} To get Figure \ref{MeanMERW}, we performed $N_P=500$ samples of the MERW approximation with a pyramid size of $n=500$. 

The parameters used in Algorithm \ref{Algo} were $d=30$ for the depth and $N=100$ for the number of sample paths generated to obtain the estimation (\ref{estimate}). 

We then computed for each box $(i,j)$ the number of times (and subsequently the proportion) that the box appears in the final pyramid, and we drew the heatmap of these frequencies. 

Regarding Figure \ref{fig:MERW}, we simply drew one pyramid of size $n=500$ with parameter $d=10$ and sample size $N=100$.

\begin{figure}[H]
	\centering
	\includegraphics[width=0.9\textwidth]{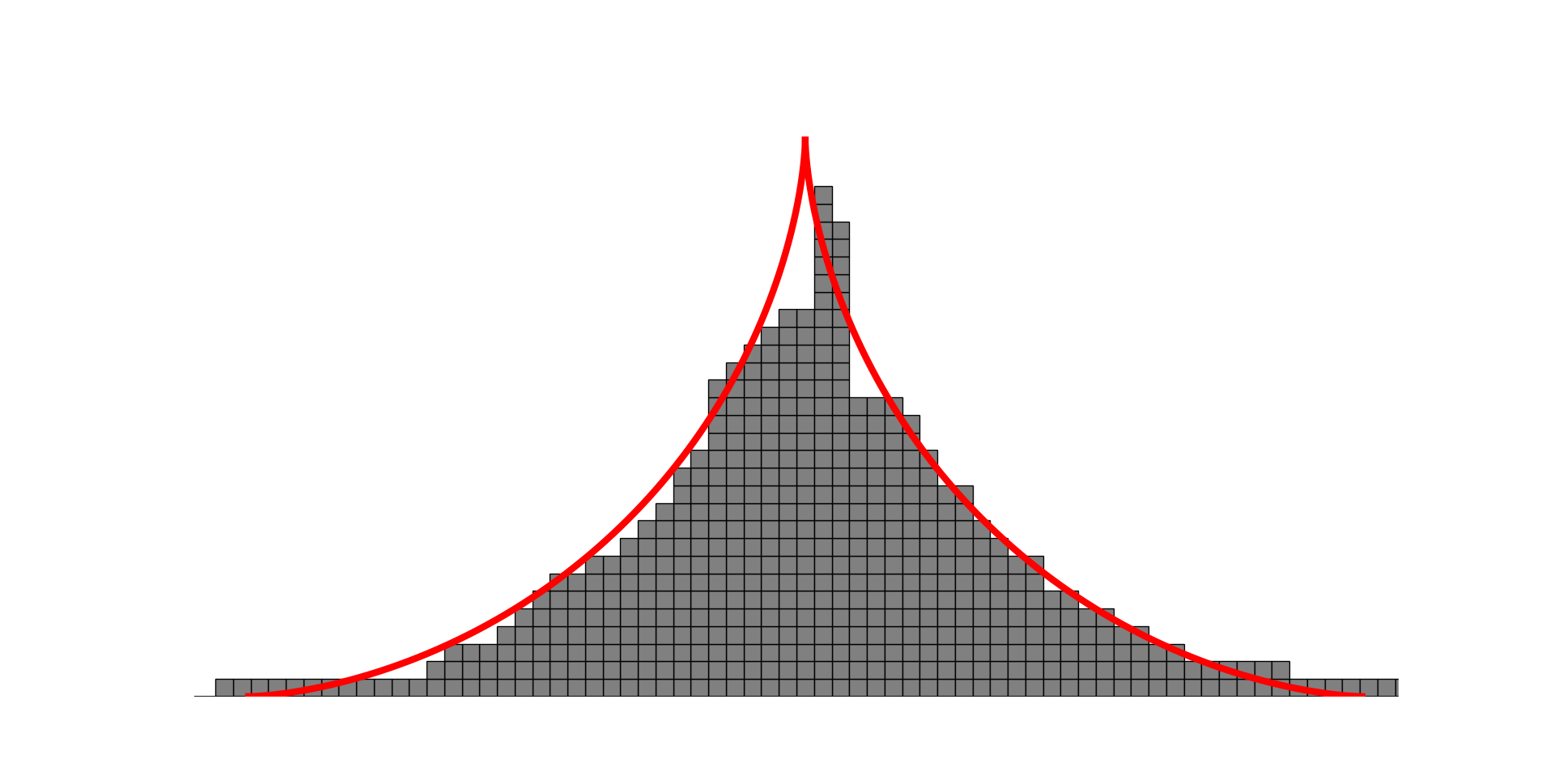}
	\captionsetup{width=0.8\textwidth}
	\caption{Pyramid size: $n=500$. Depth: $d=10$. Sample size $N=100$.}
	\label{fig:MERW}
\end{figure}

In each of these graphics, we scaled the boxes by a factor $1/\sqrt{n}$ in such a way that the area equals one. Additionally, we superposed on these figures the graph $\mathcal{G}$, the limit shape we expect for this stochastic process.

This is a symmetric version (with respect to the $y$-axis) of the classical limit shape of a Young tableau under the Plancherel measure. See \cite[p. 699]{Romik} for more details. To be more precise, introduce $\Omega : [-1,1] \to (0,\infty)$ defined by 
\begin{equation}\label{shape}
\Omega(u) = \frac{2}{\pi} \left[u \arcsin\left(u\right) + \sqrt{1-u^2}\right].
\end{equation}

Let $\mathcal{H}$ be the graph of $\Omega$ and consider $T : (u,v) \mapsto \left(\frac{u+v}{2}, \frac{v-u}{2}\right)$ and $S : (x,y) \mapsto (-x, y)$. Then, one has  
$\mathcal{G} = T(\mathcal{H}) \cup ST(\mathcal{H}).$ In light of our simulations, and similarly to the Plancherel growth process, we propose the following conjecture:
\begin{conj}
	Let $(P_n)_{n \geq 0}$ be the suitably scaled MERW Pyramidal process. The distance between the boundary of ${P_n}$ and $\mathcal{G}$ tends to zero with probability one. 
\end{conj}

\begin{acknowledgments}
This work has been supported by the EIPHI Graduate school (contract "ANR-17-EURE-0002") and by the Région "Bourgogne Franche-Comté"

\end{acknowledgments}

\section{Appendix}
\label{appendix}   

\setcounter{equation}{0}

Our goal is to slightly extend and present the well-established results regarding central measures on Weighted Bratteli Diagrams (WBDs), as stated in \cite{kerov}, in the setting of Section \ref{sec:WBD}, where {\bf countably infinite} level sets are allowed. 

Assumption \ref{Ass2} will be made throughout this section.

\subsection{One-to-one Correspondences}
A probability measure $\mu$ on $(\mathcal{T}, \mathcal{A})$ whose support is the whole BD is central if, for all $y \in \mathbb{X}$ and $s, t \in \{\varnothing \to y\}$, it holds that $\mu(C_s)w_t = \mu(C_t)w_s$. We recall that $\{\varnothing \to y\}$ denotes the set of paths starting from the root and ending at $y$, all of them being of length $|y|$.

 For such measure, the transition probabilities, defined for all $n\geq 0$, $x\in\mathbb X_n$, and $y\in\mathbb X_{n+1}$ such that $x\nearrow y$, by
\begin{equation}
p(x,y) := \frac{\mu(C_{txy})}{\mu(C_{tx})},
\end{equation}
are independent of the choice of $t \in \mathcal{T}_{n-1}$ whenever $txy \in \mathcal{T}_{n+1}$. 

This defines a Markov kernel on the BD, and the associated Markov chain $(X_n)_{n\geq 0}$, which follows the distribution $\mu$ when it starts from the root, is termed a central Markov chain. 

Furthermore, the correspondence between central measures and Positive Harmonic Functions (PHFs), that is a positive function on $\mathbb{X}$ satisfying (\ref{harmonicc}),  is established by setting
\begin{equation}\label{corres1}
\varphi(x) = \frac{\mu(C_{ux})}{\omega_{ux}},
\end{equation}
for any $n \geq 0$, $x \in \mathbb{X}_n$, and $u \in \mathcal{T}_{n-1}$. 

Conditionally on the starting point $x$, the probability of any path $s \in \{x \to y\}$ depends solely on its weight and the endpoints $x$ and $y$. It is given by
\begin{equation}
\mathbb P(X_{n}=y,\cdots,X_1=s_1| X_0=s_0) = w_s \frac{\varphi(y)}{\varphi(x)},
\end{equation}
where $n=n_y-n_x$ denotes the length of $s$ and $s=s_0\cdots s_n$ with $s_0=x$ and $s_n=y$. 

We slightly extend the definition of a central measure.

\begin{defi}\label{defmerw}	
	A measure $\mu$ on $(\mathcal{T}, \mathcal{A})$ is a saturated central measure if its support
	\begin{equation}\label{supportcentral}
	\mathbb{S} = \{x \in \mathbb{X} : \exists n \geq 0,\, \exists u \in \mathcal{T}_{n-1},\; \mu(C_{ux}) > 0\},
	\end{equation}
	is a Saturated Sub-Bratteli Diagram (SSBD) of $\mathbb{X}$, in the sense of Definition \ref{defbrattelisat}, and if its restriction to this BD is a usual central measure. 
	
	A random walk associated with a saturated central measure will be referred to as a saturated central Markov chain.
\end{defi}

\begin{prop}\label{oneone}
	There exists a one-to-one correspondence between saturated central measures  and Non-Negative Harmonic Functions (NNHFs), as described in Definition \ref{defsat}. 
\end{prop}

\begin{proof}
Given an NNHF $\varphi$, we get that $\mathbb{S} = \{x \in \mathbb{X} : \varphi(x) > 0\}$ is an SSBD of $\mathbb{X}$, as shown in the proof of Proposition \ref{suppssbd}. The restriction of $\varphi$ to this BD defines a PHF and thus a usual central measure on $\mathbb{S}$. 

Conversely, let $\mu$ be a saturated central measure on an SSBD $\mathbb{S}$ of $\mathbb{X}$, and let $\varphi : \mathbb{S} \longrightarrow ]0, \infty[$ be the corresponding PHF. Then, if we extend $\varphi$ to the whole BD $\mathbb{X}$ by setting $\varphi(x) = 0$ for all $x \in \mathbb{X} \setminus \mathbb{S}$, we obtain an NNHF for which $\mathbb{S} = \{x \in \mathbb{X} : \varphi(x) > 0\}$. 

One can easily check that these two maps are bijective and inverses of each other.
\end{proof}

\begin{rem}
Given a NNHF $\varphi$ and $n\geq 0$,  one has the general combinatorial identity 
	\begin{equation}\label{combipower}
	\sum_{y \in \mathbb{X}_n} d(\varnothing, y) \varphi(y) = 1.
	\end{equation}
We recall that the combinatorial dimension $d(\varnothing, y)$ is defined in Section \ref{sec:WBD}. It is simply the weighted number of trajectories from $\varnothing$ to $y$.
\end{rem}

\subsection{Martin boundary representation}

For this section, we continue to follow \cite{kerov} but also take inspiration from the Martin boundary construction presented in \cite{sawyer}. 

First, introduce the Martin kernel, defined for all $x, z \in \mathbb{X}$ by
\begin{equation}\label{majkernel}
K(x, z) := \frac{d(x, z)}{d(\varnothing, z)} \leq C_x := \frac{1}{d(\varnothing, x)}.
\end{equation}

Thereafter, consider  a summable family of positive numbers $(\varepsilon_x)_{x \in \mathbb{X}}$ and introduce the distance on $\mathbb X$ defined  by
\begin{equation}\label{metric}
\rho(y, z) = \sum_{x \in \mathbb{X}} \varepsilon_x \frac{|K(x, y) - K(x, z)|+|\delta_{x,y}-\delta_{x,z}|+|\delta_{n_x,n_y}-\delta_{n_x,n_z}|}{C_x + 1}.
\end{equation} 
We recall that $\delta_{a, b}$ is the usual Kronecker $\delta$-symbol, which equals $1$ when $a = b$ and $0$ otherwise, and $n_x \in \mathbb{N}$ denotes the level set number of the BD to which $x$ belongs. 

The completion of $\mathbb{X}$ with respect to the metric $\rho$ will be denoted by $\overline{\mathbb{X}}$, and the completion of each $\mathbb{X}_n$ by $\overline{\mathbb{X}}_n$. The $\delta$-terms in (\ref{metric}) ensure that if a sequence $(y_n)_{n \geq 0}$ in $\mathbb{X}$ converges to some $y \in \overline{\mathbb{X}}_n$, then it is either ultimately constant with $y \in \mathbb{X}_n$, or ultimately belongs to the $n$th level set with $y \notin \mathbb{X}_n$.

From the upper bounds in (\ref{majkernel}), we easily deduce that $\overline{\mathbb{X}}$ is a compact set. Moreover, it can be observed that $\bigsqcup_{n \in \mathbb{N}} \overline{\mathbb{X}}_n$ is an open subset of $\overline{\mathbb{X}}$.
 Consequently, the boundary of the BD,
\begin{equation}
\partial \mathbb{X} = \overline{\mathbb{X}} \setminus \bigsqcup_{n \in \mathbb{N}} \overline{\mathbb{X}}_n,
\end{equation}
is a non-empty compact set. 

An infinite path $t \in \mathcal{T}$ is regular if it converges to a point $\zeta \in \partial \mathbb{X}$.  In such cases, we define
\begin{equation}\label{harmext}
\varphi_{\zeta}(x) = \lim_{n \to \infty} K(x,t_n)= \lim_{n \to \infty} \frac{d(x, t_n)}{d(\varnothing, t_n)}.
\end{equation}
It follows that \(\varphi_\zeta(\varnothing) = 1\) and the function \(\zeta \longmapsto \varphi_\zeta(x)\) is continuous. Sometimes, for greater clarity, the function $\varphi_\zeta$ is also denoted as $K(\cdot, \zeta)$.

 Furthermore, applying the dominated convergence theorem with the help of Assumption \ref{Ass2} and the upper bound (\ref{majkernel}), we deduce from
\begin{equation}\label{pseudoharmonic}
K(x, z) = \sum_{x \nearrow y} w(x, y) K(y, z) + C_x \delta_{x, z},
\end{equation}
that \(\varphi_\zeta\) is an NNHF.
We shall denote by $\mu_\zeta$ the corresponding saturated central probability measure, as given by Proposition \ref{oneone}.

Let \(\mathcal{T}^{(k)}\) be the set of infinite paths starting from the level set \(\mathbb{X}_k\), and define the random variables \(\Theta_k : \mathcal{T} \rightarrow \mathcal{T}^{(k)}\) by \(\Theta_k(t) = (t_n)_{n \geq k}\). Denote by \(\mathcal{G}_n\) the decreasing \(\sigma\)-field generated by \(\Theta_k\), \(k \geq n\), and by \(\mathcal{G}_\infty\) the tail \(\sigma\)-field \(\bigcap_{n \geq 0} \mathcal{G}_n\).
\begin{defi}
A saturated central probability measure \(\mu\) on \((\mathcal{T}, \mathcal{A})\) is ergodic if for every \(A \in \mathcal{G}_\infty\) one has \(\mu(A) \in \{0,1\}\).
\end{defi}

\begin{thm}\label{convcentral}
The ergodic saturated central measures are precisely the $\mu_\zeta$, $\zeta \in \partial \mathbb{X}$, defined above. These measures are the extremal points of the convex set of saturated central probability measures, and each saturated central probability measure can be represented as the Choquet integral 
\begin{equation}\label{choquet}
\mu = \int_{\partial \mathbb{X}} \mu_\zeta\, m(d\zeta),
\end{equation}
where \(m\) is a probability distribution on \(\partial \mathbb{X}\). Correspondingly, the associated NNHF is given by
\begin{equation}
\varphi(x) = \int_{\partial \mathbb{X}} \varphi_\zeta(x)\, m(d\zeta).
\end{equation}

Let \((X_n)_{n \geq 0}\) be the saturated central Markov chain associated with \(m\). Denote by \(\mathbb{S}\) its support and by $\mathbb{P}_x$ the distribution starting from $x \in \mathbb{S}$. Then, for all \(x \in \mathbb{S}\),
\begin{equation}\label{exit}
\lim_{n \to \infty} X_n = Z \quad \mathbb{P}_x\text{-a.s.}, \quad \text{with} \quad Z \sim \frac{\varphi_\zeta(x)}{\varphi(x)} m(d\zeta).
\end{equation}
\end{thm}

\begin{proof}
	Let $\mu$ be an arbitrary saturated central probability measure.

	For all \(k \geq 0\), \(s \in \mathcal{T}_{k-1}\), and $x \in \mathbb{X}_k$ such that \(sx \in \mathcal{T}_k\), it can be checked that for all \(t \in \mathcal{T}\) and \(n \geq k\),
	\begin{equation}
	\mu(C_{sx}|\mathcal{G}_n)(t) = \omega_{sx}\frac{d(x, t_n)}{d(\varnothing, t_n)}.
	\end{equation}
	
	Applying the backward martingale convergence theorem, we get for \(\mu\)-almost all \(t \in \mathcal{T}\),
	\begin{equation}
	\mu(C_{sx}|\mathcal{G}_\infty)(t) = \omega_{sx}\lim_{n \to \infty} \frac{d(x, t_n)}{d(\varnothing, t_n)}.
	\end{equation}

We obtain that \(\mu\)-almost all trajectories $t \in \mathcal{T}$ are regular, as defined in (\ref{harmext}). 

Thus, one can introduce a $\mathcal G_\infty$-measurable random variable \(Z : \mathcal{T} \rightarrow \partial \mathbb{X}\) corresponding to the limit point of each trajectory. Since for any $\zeta \in \partial \mathbb{X}$ one can write
\begin{equation}
\mu_\zeta(C_{sx}) = \omega_{sx} \varphi_\zeta(x) = \omega_{sx} \lim_{n \to \infty} \frac{d(x, t_n)}{d(\varnothing, t_n)},
\end{equation}
we deduce that 
\begin{equation}
\mu(C_{sx}|\mathcal{G}_\infty)(t) = \mu_{Z(t)}(C_{sx}) \quad \mu(dt)\text{-a.s.}.
\end{equation}

As a direct consequence, we obtain the representation (\ref{choquet}), in which \(m\) represents the distribution of the exit point \(Z\) in the Martin boundary when the corresponding saturated central Markov chain starts from the root. 

We also directly obtain (\ref{exit}) when \(x = \varnothing\), and the result for arbitrary \(x\) can be deduced simply by standard conditional computation.

If \(\mu\) is ergodic, we obtain that \(Z\) is constant \(\mu\)-almost surely, and thus there exists \(\zeta \in \partial \mathbb{X}\) such that \(\mu = \mu_\zeta\). Conversely, if \(\mu\) is not ergodic, there exists \(A \in \mathcal{G}_\infty\) such that \(\mu(A)\) and \(\mu(A^c)\) are both positive. Then, consider
	\begin{equation}
	\mu(\cdot | A) = \frac{1}{\mu(A)}\int_{A}\mu_{Z(t)}(\cdot)\mu(dt) \quad \text{and} \quad \mu(\cdot | A^c) = \frac{1}{\mu(A^c)}\int_{A^c}\mu_{Z(t)}(\cdot)\mu(dt).
	\end{equation}
	
These are two saturated central probability measures for which the corresponding exit points belong respectively to \(A\) and \(A^c\) with probability one. 

Therefore, these distributions are mutually singular, indicating that \(Z\) is not \(\mu\)-almost surely constant. This completes the proof.
\end{proof}

\bibliography{biblio}
\bibliographystyle{unsrt}

\end{document}